\documentclass[11pt]{amsart}

\usepackage[utf8]{inputenc}
\usepackage[T1]{fontenc}

\usepackage{amsmath,amssymb,amsfonts,textcomp,amsthm,xifthen,graphicx,color,pgfplots}
\usepackage{ifthen}

\usepackage{enumerate}
\usepackage{enumitem}

\usepackage{fullpage}

\usepackage[utf8]{inputenc}

\usepackage[pdftex,
            pdfauthor={Thomas F\"uhrer and Michael Karkulik},
            pdftitle={Least-squares FEMs for parabolic problems},
            ]{hyperref}

\usepackage{pgfplots}
\usepgfplotslibrary{external}
\usepgfplotslibrary{colorbrewer}
\usepackage[capitalize,nameinlink]{cleveref}[0.19]
\crefname{section}{section}{sections}
\crefname{subsection}{subsection}{subsections}
\Crefname{section}{Section}{Sections}
\Crefname{subsection}{Subsection}{Subsections}
\Crefname{figure}{Figure}{Figures}
\crefformat{equation}{\textup{#2(#1)#3}}
\crefrangeformat{equation}{\textup{#3(#1)#4--#5(#2)#6}}
\crefmultiformat{equation}{\textup{#2(#1)#3}}{ and \textup{#2(#1)#3}}
{, \textup{#2(#1)#3}}{, and \textup{#2(#1)#3}}
\crefrangemultiformat{equation}{\textup{#3(#1)#4--#5(#2)#6}}%
{ and \textup{#3(#1)#4--#5(#2)#6}}{, \textup{#3(#1)#4--#5(#2)#6}}{, and \textup{#3(#1)#4--#5(#2)#6}}
\Crefformat{equation}{#2Equation~\textup{(#1)}#3}
\Crefrangeformat{equation}{Equations~\textup{#3(#1)#4--#5(#2)#6}}
\Crefmultiformat{equation}{Equations~\textup{#2(#1)#3}}{ and \textup{#2(#1)#3}}
{, \textup{#2(#1)#3}}{, and \textup{#2(#1)#3}}
\Crefrangemultiformat{equation}{Equations~\textup{#3(#1)#4--#5(#2)#6}}%
{ and \textup{#3(#1)#4--#5(#2)#6}}{, \textup{#3(#1)#4--#5(#2)#6}}{, and \textup{#3(#1)#4--#5(#2)#6}}
\crefdefaultlabelformat{#2\textup{#1}#3}

\newcommand{\ts}{{\ensuremath{k}}}

\newtheorem{theorem}{Theorem}
\newtheorem{lemma}[theorem]{Lemma}
\newtheorem{corollary}[theorem]{Corollary}
\newtheorem{proposition}[theorem]{Proposition}
\newtheorem{remark}[theorem]{Remark}

\newcommand{\err}{\operatorname{err}}

\newcommand{\blf}{b} 
\newcommand{\blfTot}{a} 
\newcommand{\blfns}{\widetilde\blf} 

\newcommand{\ep}{\mathcal{E}}

\def\eps{\varepsilon}

\def\enorm#1{|\hspace*{-.5mm}|\hspace*{-.5mm}|#1|\hspace*{-.5mm}|\hspace*{-.5mm}|}

\newcommand{\ip}[2]{(#1\hspace*{.5mm},#2)}

\newcommand{\norm}[3][]{#1\|#2#1\|_{#3}}

\newcommand{\diam}{\mathrm{diam}}
\newcommand{\wilde}{\widetilde}

\def\div{{\rm div\,}}

\newcommand{\Hdivset}[1]{\boldsymbol{H}(\div;#1)}

\newcommand\CF{C_\mathrm{F}} 

\newcommand{\set}[2]{\big\{#1\,:\,#2\big\}}

\newcommand{\pwnabla}{\nabla_\TT}

\newcommand{\RT}{\ensuremath{\mathcal{RT}}}

\newcommand{\R}{\ensuremath{\mathbb{R}}}
\newcommand{\N}{\ensuremath{\mathbb{N}}}
\newcommand{\HH}{\ensuremath{{\boldsymbol{H}}}}

\newcommand{\LL}{\ensuremath{\boldsymbol{L}}}

\newcommand{\vv}{\ensuremath{\boldsymbol{v}}}
\newcommand{\ww}{\ensuremath{\boldsymbol{w}}}
\newcommand{\TT}{\ensuremath{\mathcal{T}}}

\newcommand{\cS}{\ensuremath{\mathcal{S}}}

\newcommand{\PP}{\ensuremath{\mathcal{P}}}
\newcommand{\OO}{\ensuremath{\mathcal{O}}}

\newcommand{\Amat}{\ensuremath{\boldsymbol{A}}}
\newcommand{\bbeta}{\ensuremath{\boldsymbol{\beta}}}

\newcommand{\ssigma}{{\boldsymbol\sigma}}
\newcommand{\ttau}{{\boldsymbol\tau}}

\newcommand{\cchi}{{\boldsymbol\chi}}

\newcommand{\uu}{\boldsymbol{u}}

\newcounter{constantsnumber}
\def\setc#1{
  \ifthenelse{\equal{#1}{poinc}}{C_{\rm edge}}{ 
   \refstepcounter{constantsnumber}
   \label{const#1}C_{\theconstantsnumber}}}

\def\c#1{
  \ifthenelse{\equal{#1}{poinc}}{C_{\rm edge}}{ 
    C_{\ref{const#1}}}}

\begin{document}

\title[]{
New a priori analysis of first-order system least-squares finite element methods for parabolic problems}
\date{\today}
\author{Thomas F\"{u}hrer}
\address{Facultad de Matem\'{a}ticas, Pontificia Universidad Cat\'{o}lica de Chile, Santiago, Chile}
\email{tofuhrer@mat.uc.cl \emph{(corresponding author)}}

\author{Michael Karkulik}
\address{Departamento de Matem\'{a}tica,
Universidad T\'{e}cnica Federico Santa Mar\'{i}a, Valpara\'{i}so, Chile}
\email{michael.karkulik@usm.cl}

\thanks{{\bf Acknowledgment.} 
This work was supported by CONICYT through FONDECYT projects 11170050 ``Least-squares methods for obstacle problems''
(TF) and 1170672 ``Fast space-time discretizations for fractional evolution equations'' (MK)}

\begin{abstract}
We provide new insights into the a priori theory for a time-stepping scheme based on least-squares
finite element methods for parabolic first-order systems.
The elliptic part of the problem is of general reaction-convection-diffusion type.
The new ingredient in the analysis is an elliptic projection operator defined via a non-symmetric bilinear form, although
the main bilinear form corresponding to the least-squares functional is symmetric.
This new operator allows to prove optimal error estimates in the natural norm associated to the problem and, under
additional regularity assumptions, in the $L^2$ norm.
Numerical experiments are presented which confirm our theoretical findings.
\end{abstract}
\maketitle

\section{Introduction}

In this work we analyse a time-stepping scheme based on least-squares finite element methods for a
first-order reformulation of the evolution problem
\begin{subequations}\label{eq:model}
\begin{alignat}{2}
  \dot{u}-\div \Amat\nabla u -\bbeta\cdot\nabla u + \gamma u &=f &\quad&\text{in } (0,T)\times\Omega, \label{eq:model:pde} \\
  u &= 0 &\quad&\text{in } (0,T)\times \partial\Omega, \\
  u(0,\cdot) &= u^0 &\quad&\text{in } \Omega.
\end{alignat}
\end{subequations}
The coefficients $\Amat\in L^\infty(\Omega)_\mathrm{sym}^{d\times d}$, 
$\bbeta\in L^\infty(\Omega)^d$, $\gamma\in L^\infty(\Omega)$ are assumed to satisfy
\begin{align}\label{eq:coeff}
  L^\infty(\Omega)\ni\tfrac12\div\bbeta+\gamma \geq 0, \quad\text{and}\quad 
  \lambda_\mathrm{min}(\Amat) \geq \alpha_0 > 0 \quad\text{a.e. in } \Omega.
\end{align}
This ensures that~\cref{eq:model} admits a unique solution. (Here, $\lambda_\mathrm{min}$ denotes the minimal
eigenvalue.)
We refer to~\cite{evans} for the analytical treatment and to~\cite{thomee} for an overview and analysis of different
Galerkin methods.
For a simpler presentation we suppose that the coefficients are independent of $t$, however, we stress 
that our analysis also applies in the case of time-dependent coefficients, if~\cref{eq:coeff} holds uniformly in
time.

Discretizations by finite elements in space and finite differences in time (also called time-stepping) are attractive for
the numerical approximation of solutions of evolutions equations such as~\cref{eq:model}. This is due to the fact
that such discretizations are much simpler to implement then discretizations of the space-time domain.
In the context of least-squares finite element methods such time-stepping schemes have been analysed
in, e.g.,~\cite{SplitLSQ,YangI,YangII}. The analysis in these works include the Euler and Crank-Nicolson scheme in time.
Another time-stepping approach is considered in~\cite{starkeParabolicI,starkeParabolicII}.
We refer also to~\cite[Chapter~9]{BochevGunzburgerLSQ} for an overview on the whole topic.
In accordance with the last reference, cf.~\cite[Section~2.2]{BochevGunzburgerLSQ}, the first step to obtain a
practical least-squares discretization is to work with a first-order reformulation of the PDE at hand. This way,
high regularity conditions on the discrete space are avoided, and condition numbers of underlying systems are kept at
a reasonable level. If the first-order reformulation of~\cref{eq:model} allows for the so-called \emph{splitting property},
optimal a priori error estimates are known, cf.~\cite{SplitLSQ,BochevGunzburgerLSQ}.
It can be easily shown~\cite{SplitLSQ} that the method then reduces to a time-stepping Galerkin method in the primal variable,
and optimal a priori error estimates can be extracted from the classical reference~\cite[Chapter~1]{thomee}.
We refer to~\cref{section:decoupling} for more details.
This splitting property is quite specific and in general only holds if $\bbeta=0=\gamma$ in~\cref{eq:model}.
Available a priori analysis of least-squares finite element time stepping schemes does not extend to the general case
$(\bbeta,\gamma)\neq (0,0)$. In fact, if $(\bbeta,\gamma)\neq (0,0)$, current literature provides optimal a priori error estimates
under severe and impractical restrictions, e.g., the use of different meshes for approximations of 
the variables of the first-order system.

The purpose of the work at hand is to provide optimal error estimates in the $L^2$ norm for the scalar variable and in
the natural norm associated with the problem, in particular in the case $(\bbeta,\gamma)\neq (0,0)$.
Since the pioneering work of Mary Wheeler~\cite{wheeler},
one of the main tools in the a priori analysis of time-stepping Galerkin schemes is the \emph{elliptic projection operator},
defined as the discrete solution operator on the \textit{spatial part} of the parabolic PDE.
The main tool in the following analysis is a new \emph{elliptic projection operator}, which is defined with respect to only
part of the spatial terms of the overall bilinear form. In particular, these terms constitute a non-symmetric bilinear form.
At first glance this seems to be somehow counter-intuitive, since the overall bilinear form associated to the least-squares
method is symmetric.
Nevertheless, it turns out that this is the right choice, and it allows us to mimic some of the main ideas of
well-known proofs for Galerkin finite element methods.
Moreover, for the proof of optimal $L^2$ error estimates we show, under some regularity assumptions, that the $L^2$
error between a sufficient regular function and its elliptic projection is of higher order.
This is done by considering solutions corresponding to problems that involve the aforementioned non-symmetric bilinear form
and by using duality arguments.
We stress that everything has to be analized carefully due to the fact that the bilinear form depends on the (arbitrarily
small) temporal step-size.

To keep presentation simple, we consider a backward Euler scheme for time discretization. We stress that we have no
reason to believe that our results do not extend correspondingly to a higher-order discretization in time (e.g., Crank-Nicolson).
Likewise, the main ideas of our analysis are independent of the particular choice of (spatial) finite element spaces.
However, to obtain convergence orders in terms of powers of the mesh-width $h$ we stick to standard conforming
finite element spaces (piecewise polynomials for the scalar variable and Raviart-Thomas elements for the vector-valued
variable) on the same simplical mesh.

\subsection{Outline}
The remainder of this work is given as follows: In~\cref{sec:main}, we introduce our model problem and
discrete method and discuss necessary definitions and results needed throughout our work.
In particular, in~\cref{sec:duality}, we introduce and analyse the elliptic projection operator.
The~\cref{sec:apriori} contains the main results (stability of the discrete method and optimal convergence
rates in $L^2$ and energy norm) and their proofs. In~\cref{sec:ext} we show how to extend our ideas
to a different problem. Finally, in~\cref{sec:num} we present numerical examples.

\section{Least-squares methods}\label{sec:main}

\subsection{Notation}
Let $\Omega\subset \R^d$ denote a bounded Lipschitz domain with polygonal boundary $\Gamma := \partial \Omega$.
Let $\LL^2(\Omega) := L^2(\Omega)^d$.
We use the usual notation for Sobolev spaces $H_0^1(\Omega)$, $H^s(\Omega)$, $\HH^s(\Omega):= H^s(\Omega)^d$,
$\Hdivset\Omega : = \set{\ssigma \in\LL^2(\Omega)}{\div\ssigma\in L^2(\Omega)}$.
Recall that $\norm{u}{} \leq \CF \norm{\nabla u}{}$ for all $u\in H_0^1(\Omega)$, where $\norm\cdot{}$ denotes the
$L^2(\Omega)$ norm. 
Moreover, the $L^2(\Omega)$ scalar product is denoted by $\ip\cdot\cdot$.

We work with the space
\begin{align*}
  U := H_0^1(\Omega) \times \Hdivset\Omega
\end{align*}
equipped with the $\ts$-dependent norm
\begin{align*}
  \norm{(u,\ssigma)}\ts^2 := \norm{\nabla u}{}^2 + \norm{\ssigma}{}^2 + \ts\norm{\div\ssigma}{}^2.
\end{align*}
The parameter $\ts>0$ corresponds to the time step later on.

Let $0=:t_0<t_1<\dots<t_N:=T$ denote an arbitrary partition of the time interval $[0,T]$. We call $\ts_n :=
t_n-t_{n-1}$ a time step.
For a simpler notation, we use upper indices to indicate time evaluations, i.e., for a time-dependent function $v =
v(t;\cdot)$ we use the notation $v^n(\cdot) := v(t_n;\cdot)$.
Moreover, we use similar notations for discrete quantities which however are not defined for all $t\in [0,T]$. 
For instance, later on $u_h^n$ will denote an approximation of $u^n = u(t_n;\cdot)$ where $u$ is the exact solution
of~\cref{eq:model}.

Operators involving derivatives with respect to the spatial variables are denoted by $\nabla$ and $\div$, whereas 
the first and second derivate with respect to the time variable are denoted by $\dot{(\cdot)}$, $\ddot{(\cdot)}$.
For vector-valued functions the time derivative is understood componentwise.

We write $A\lesssim B$ if there exists $C>0$ such that $A\leq C B$ and $C$ is independent of the quantities of interest.
Analogously we define $A\gtrsim B$. If both $A\lesssim B$ and $B\lesssim A$ hold, then we write $A\simeq B$.
In this work we use these notations if the involved constants only depend on the coefficients $\Amat,\bbeta,\gamma$, the
final time $T$, $\Omega$, some fixed polynomial degree $p\in\N_0$ of the discrete spaces, and shape-regularity of the
underlying triangulation.

\subsection{First-order system and least-squares method}
Let $u$ denote the exact solution of~\cref{eq:model}. 
Introducing $\ssigma:= \Amat\nabla u$, we rewrite~\cref{eq:model} as the equivalent first-order system
\begin{subequations}\label{eq:fo}
  \begin{alignat}{2}
    \dot{u} - \div\ssigma -\bbeta\cdot\nabla u + \gamma u &= f &\quad&\text{in } (0,T)\times\Omega, \label{eq:fo:a} \\
    \ssigma - \Amat\nabla u &= 0 &\quad&\text{in } (0,T)\times\Omega, \label{eq:fo:b} \\
    u &= 0 &\quad&\text{on } (0,T)\times\Gamma, \label{eq:fo:c} \\
    u(0,\cdot) &= u^0 &\quad&\text{in }\Omega. \label{eq:fo:d}
\end{alignat}
\end{subequations}

We approximate $\dot{u}^n$ by a backward difference, i.e., $\dot{u}^n \approx (u^n-u^{n-1})/\ts_n$.
This approach naturally leads to a time-stepping method:
given an approximation $\wilde u^{n-1}$ of $u^{n-1}$, we
consider the problem of finding $(u,\ssigma)\in U = H_0^1(\Omega)\times\Hdivset\Omega$ such that
\begin{subequations}\label{eq:model:euler}
\begin{align}
  \frac{u-\wilde u^{n-1}}{\ts_n} - \div\ssigma -\bbeta\cdot\nabla u + \gamma u &= f^n, \\
  \ssigma - \Amat\nabla u &= 0
\end{align}
\end{subequations}
for $n=1,\dots,N$. The last equations will be put into a variational scheme using a least-squares approach.
For given data $g,w\in L^2(\Omega)$ we consider the functional
\begin{align}\label{eq:def:lsqfun}
  J_n(\uu;g,w) := \ts_n \norm{\frac{u-w}{\ts_n}- \div\ssigma -\bbeta\cdot\nabla u + \gamma u-g}{}^2
  + \norm{\Amat^{1/2}\nabla u-\Amat^{-1/2}\ssigma}{}^2.
\end{align}
for $\uu = (u,\ssigma)\in U$.
Note that~\cref{eq:model:euler} is an elliptic equation since $\ts_n>0$ and therefore admits a unique solution that
can also be characterized as the unique minimizer of the problem
\begin{align}\label{eq:min:euler}
  \min_{\vv\in U} J_n(\vv;f^n,\wilde u^{n-1}).
\end{align}
It is well-known~\cite{CLMMcC1994} that the functional satisfies the norm equivalence
\begin{align*}
  \norm{\nabla u}{}^2 + \norm{\ssigma}{}^2 + \norm{\div\ssigma}{}^2 \simeq J_n(\uu;0,0)
\end{align*}
for all $\uu = (u,\ssigma)\in U$. 
In particular, standard textbook knowledge of the least-squares methodology, see~\cite{BochevGunzburgerLSQ}, shows that
there exists a unique minimizer if the space $U$ in~\cref{eq:min:euler} is replaced by some closed subspace.
Note that the norm equivalence constants depend on the time step $\ts_n$ in general.

We introduce the bilinear forms $\blf_n,\blfTot_n : U\times U \to \R$, 
\begin{align*}
  \blf_n(\uu,\vv) &:= \ip{u}{-\div\ttau-\bbeta\cdot\nabla v+\gamma v} + \ip{-\div\ssigma-\bbeta\cdot\nabla u + \gamma u}{v}
  \\ &\qquad+ \ts_n \ip{-\div\ssigma-\bbeta\cdot\nabla u + \gamma u}{-\div\ttau-\bbeta\cdot\nabla v+\gamma v}
  \\ &\qquad + \ip{\Amat^{1/2}\nabla u-\Amat^{-1/2}\ssigma}{\Amat^{1/2}\nabla v-\Amat^{-1/2}\ttau}, \\
  \blfTot_n(\uu,\vv) &:= \frac1{\ts_n}\ip{u}v + \blf_n(\uu,\vv)
\end{align*}
for all $\uu=(u,\ssigma),\vv=(v,\ttau)\in U$.
Moreover, define for some given $w\in L^2(\Omega)$, the functional $F_n : U\to \R$ by
\begin{align*}
  F_n(\vv;f^n,w) := \ts_n\ip{f^n+\frac{w}{\ts_n}}{\frac{v}{\ts_n}-\div\ttau-\bbeta\cdot\nabla v+\gamma v}
\end{align*}
for $\vv=(v,\ttau)\in U$.

The \emph{backward Euler} scheme reads as follows: For all $n=1,\dots,N$ let $U_h^n \subseteq U$ denote a closed
subspace. Let $u_h^0 \in L^2(\Omega)$ be some approximation of the initial data, i.e., $u_h^0\approx u^0$.
For $n=1,\dots,N$ we seek solutions $\uu_h^n = (u_h^n,\ssigma_h^n)\in U_h^n$ of
\begin{align}\label{eq:euler}
  \blfTot_n(\uu_h^n,\vv_h) = F_n(\vv_h;f^n,u_h^{n-1}) \quad\text{for all } \vv_h = (v_h,\ttau_h)\in U_h^n.
\end{align}

From our discussion above we conclude:
\begin{theorem}\label{thm:euler}
  For all $n=1,\dots,N$ Problem~\cref{eq:euler} admits a unique solution $\uu_h^n\in U_h^n$.
  \qed
\end{theorem}

\subsection{Problem decoupling for convection-reaction free problems}\label{section:decoupling}
In the case of convection and reaction free problems, i.e., $(\bbeta,\gamma)=(0,0)$, 
one can show, see~\cite{BochevGunzburgerLSQ,SplitLSQ}, that the backward Euler least-squares method reduces to a simplified problem:
Suppose $(\bbeta,\gamma)=(0,0)$ and that $\Amat$ is the identity. The bilinear form then reads
\begin{align*}
  \blfTot_n(\uu,\vv) &= \frac1{\ts_n}\ip{u}v + \ip{u}{-\div\ttau} + \ip{-\div\ssigma}{v}
  + \ts_n \ip{-\div\ssigma}{-\div\ttau} + 
  \ip{\nabla u-\ssigma}{\nabla v-\ttau} \\
  &= \frac1{\ts_n}\ip{u}v + \ip{\nabla u}{\ttau} + \ip{\ssigma}{\nabla v} + \ts_n \ip{\div\ssigma}{\div\ttau} + 
  \ip{\nabla u-\ssigma}{\nabla v-\ttau} \\
  &= \frac1{\ts_n}\ip{u}v + \ts_n \ip{\div\ssigma}{\div\ttau} + \ip{\nabla u}{\nabla v} + \ip{\ssigma}{\ttau},
\end{align*}
where we have only used integration by parts.
Testing with $\vv_h = (v_h,0)$ in~\cref{eq:euler} leads to the variational problem
\begin{align*}
  \frac1{\ts_n}\ip{u_h^n}{v_h} + \ip{\nabla u_h^n}{\nabla v_h^n} = \frac1{\ts_n} \ip{u_h^{n-1}}{v_h} +
  \ip{f^n}{v_h},
\end{align*}
which is independent of $\ssigma_h^j$ for all $j=1,\dots,n$. Hence, this is the backward Euler method for the standard
Galerkin scheme, where optimal error estimates are known~\cite{thomee}.
Therefore, the least-squares problem simplifies to a well-known method.
However, the situation is completely different when $(\bbeta,\gamma)\neq (0,0)$ 
because the problem does not decouple anymore and therefore does not reduce to a simplified method.
We stress that our analysis below is in particular also valid for $(\bbeta,\gamma)\neq (0,0)$.

\subsection{Discrete spaces and approximation properties}\label{sec:discrete}
The main ideas of our analysis are independent of specific discrete spaces. However, for the analysis of
convergence orders we will fix the following discrete setting.
By $\TT$ we denote a simplicial mesh on $\Omega$. As usual, $h$ denotes the biggest element diameter in $\TT$.
We will only consider uniform sequences of meshes, and we assume that our meshes are shape-regular.
The discrete spaces which we will use are the space $\cS_0^{p+1}(\TT)$ of globally continuous,
$\TT$-piecewise polynomials of degree at most $p+1$ satisfying homogeneous boundary conditions,
and the Raviart-Thomas space $\RT^p(\TT)$ of order $p$. We will also need the
space $\PP^p(\TT)$ of piecewise polynomials of degree at most $p$.
For the remainder of this work we consider the following discrete subspace of $U$,
\begin{align*}
  U_h := U_{h,p} := \cS_0^{p+1}(\TT) \times \RT^p(\TT),
\end{align*}
We will need certain interpolation operators mapping into the spaces.
By $I^{\rm SZ}_{h,p}:H^1_0(\Omega)\rightarrow\cS^{p+1}_0(\TT)$ we denote
the Scott-Zhang interpolation operator from~\cite{SZ_90}, by $I^{\rm
RT}_{h,p}:\Hdivset\Omega\cap\HH^1(\Omega)\rightarrow \RT^p(\TT)$
the Raviart-Thomas interpolation operator~\cite{RT_77}, and by $\Pi_{h,p}:L^2(\Omega)\rightarrow \PP^p(\TT)$ the
$L^2(\Omega)$-orthogonal projection.
We keep in mind that there the commutativity property $\div I^{\rm RT}_{h,p}\ssigma = \Pi_{h,p}\div\ssigma$.

The mesh $\TT$ is called shape-regular if
\begin{align*}
  \max_{T\in\TT} \frac{\diam(T)^d}{|T|} \leq \kappa < \infty,
\end{align*}
where $|T|$ denotes the volume of an element.

We make use of the spaces
\begin{align*}
  C^1(\TT) &:= \set{v\in L^\infty(\Omega)}{v|_T \in C^1(\overline T) \text{ for all } T\in\TT}, \\
  H^{p+1}(\TT) &:= \set{v\in L^2(\Omega)}{v|_T \in H^{p+1}(T) \text{ for all } T\in\TT},
\end{align*}
where $H^{p+1}(T)$ denotes the Sobolev space on $T$.
Moreover, let $\pwnabla : H^1(\TT) \to L^2(\Omega)$ denote the elementwise gradient operator.

The following result is a simple consequence of the approximation
properties of the interpolation operators mentioned above, respectively the commutativity property.
\begin{proposition}\label{prop:approx}
  Let $\uu = (u,\ssigma) \in U$, $p\in\N_0$. Suppose that $u\in H^{p+2}(\Omega)\cap H_0^1(\Omega)$, $\ssigma\in
  \HH^{p+1}(\Omega)$, and $\div\ssigma\in H^{p+1}(\TT)$. Then,
  \begin{align*}
    \min_{\vv_h\in U_h} \norm{\uu-\vv_h}\ts \leq C_\mathrm{app} h^{p+1} (\norm{u}{H^{p+2}(\Omega)} +
    \norm{\ssigma}{\HH^{p+1}(\Omega)} + \ts^{1/2}\norm{\div\ssigma}{H^{p+1}(\TT)})
  \end{align*}
  The constant $C_\mathrm{app}>0$ only depends on shape-regularity of $\TT$ and $p\in \N_0$.
  \qed
\end{proposition}

\subsection{Regularity of solutions}
For some given $g\in L^2(\Omega)$ we consider the second order elliptic problem
\begin{align*}
\begin{split}
  -\div\Amat\nabla v -\bbeta\cdot\nabla v + \gamma v &= g, \\
  v|_\Gamma &= 0
\end{split}
\end{align*}
and the equivalent first-order problem
\begin{subequations}
\begin{align*}
  -\div\ttau - \bbeta\cdot\nabla v + \gamma v &= g, \\
  \Amat\nabla v - \ttau &= 0, \\
  v|_\Gamma &= 0.
\end{align*}
\end{subequations}
Moreover, we consider the (dual or adjoint) problem
\begin{align*}
  -\div(\Amat\nabla w - \bbeta w) + \gamma w &= g, \\
  w|_\Gamma &= 0.
\end{align*}
Throughout, we assume that $\bbeta$ and $\gamma$ satisfy
\begin{subequations}\label{eq:regularity}
\begin{align}\label{eq:regularity:b}
  \bbeta\in C^1(\TT)^d \cap \Hdivset\Omega, \quad \gamma\in C^1(\TT).
\end{align}
For the results concerning optimal rates of the $L^2$ error we additionally assume that $\Omega$ and the coefficients
are such that (for the solutions above)
\begin{align}\label{eq:regularity:a}
  \norm{v}{H^2(\Omega)} + \norm{\ttau}{\HH^1(\Omega)} + \norm{w}{H^2(\Omega)} \lesssim \norm{g}{}.
\end{align}
\end{subequations}

\begin{remark}
  Estimate~\cref{eq:regularity:a} is satisfied for $d=2$ if $\Omega$ is convex, $\Amat$ the identity, and
  $\bbeta,\gamma$ satisfy~\cref{eq:regularity:b},
  see~\cite{grisvard} for details on the regularity of solutions on polygonal domains.
\end{remark}

\subsection{Elliptic projection and duality arguments}\label{sec:duality}
In this section we introduce and analyse the so-called \emph{elliptic projection operator} $\ep_h : U\to
U_h$, which will be used in the remainder of this work.
As discussed in the introduction, the elliptic part of the problem is described by a non-symmetric bilinear form
given by
\begin{align*}
  \blfns_n(\uu,\vv) &:= \ip{-\div\ssigma-\bbeta\cdot\nabla u + \gamma u}{v}
  \\ &\qquad+ \ts_n \ip{-\div\ssigma-\bbeta\cdot\nabla u + \gamma u}{-\div\ttau-\bbeta\cdot\nabla v+\gamma v}
  \\ &\qquad + \ip{\Amat^{1/2}\nabla u-\Amat^{-1/2}\ssigma}{\Amat^{1/2}\nabla v-\Amat^{-1/2}\ttau}
\end{align*}
for all $\uu,\vv\in U$.

We analyse basic properties of $\blf$ and $\blfns$.
For the remainder of this section we set $\ts:=\ts_n$ and skip the index $n$ in the bilinear forms.
\begin{lemma}\label{lem:blf}
  There exist constants $c_\blf,C_\blf>0$ independent of $\ts\in(0,T]$ such that
  \begin{align}
    c_\blf\norm{\uu}\ts^2 &\leq \min\{ \blf(\uu,\uu),\blfns(\uu,\uu) \}, \label{eq:blf:coer}\\
    \max\{ |\blf(\uu,\vv)|,|\blfns(\uu,\vv)|\} &\leq C_\blf \norm{\uu}\ts\norm{\vv}\ts \label{eq:blf:bound}
  \end{align}
  for all $\uu,\vv\in U$.
\end{lemma}
\begin{proof}
  We start to proof boundedness~\cref{eq:blf:bound} of $\blfns$.
  The Cauchy-Schwarz inequality, the triangle inequality and Friedrich's inequality show that
  \begin{align*}
    &\ts |\ip{-\div\ssigma-\bbeta\cdot\nabla u + \gamma u}{-\div\ttau-\bbeta\cdot\nabla v+\gamma v}|
    \\&\qquad \leq \ts^{1/2} \norm{-\div\ssigma-\bbeta\cdot\nabla u + \gamma u}{} 
    \ts^{1/2} \norm{-\div\ttau-\bbeta\cdot\nabla v+\gamma v}{} \\
    &\qquad \lesssim (\ts^{1/2} \norm{\div\ssigma}{} + \norm{\nabla u}{} + \norm{u}{}) 
    (\ts^{1/2} \norm{\div\ttau}{} + \norm{\nabla v}{} + \norm{v}{})
     \lesssim \norm{\uu}\ts\norm{\vv}\ts .
  \end{align*}
  Moreover, 
  \begin{align*}
    |\ip{\Amat^{1/2}\nabla u-\Amat^{-1/2}\ssigma}{\Amat^{1/2}\nabla v-\Amat^{-1/2}\ttau}|
    &\lesssim (\norm{\nabla u}{}+\norm{\ssigma}{})(\norm{\nabla v}{}+\norm{\ttau}{})
    \lesssim \norm{\uu}\ts\norm{\vv}\ts.
  \end{align*}
  Integration by parts yields
  \begin{align*}
    |\ip{-\div\ssigma-\bbeta\cdot\nabla u + \gamma u}{v}| &= |\ip{\ssigma}{\nabla v} - \ip{\bbeta\cdot\nabla u}v
    +\ip{\gamma u}v| \lesssim \norm{\uu}\ts\norm{\vv}\ts.
  \end{align*}
  Putting altogether shows boundedness of $\blfns$. The same arguments prove boundedness of $\blf$ and
  therefore~\cref{eq:blf:bound}. It remains to show~\cref{eq:blf:coer}.
  First, from~\cref{eq:coeff} it follows with integration by parts that
  \begin{align*}
    \ip{u}{-\bbeta\cdot\nabla u+\gamma u} = \ip{u}{(\tfrac12\div\bbeta+\gamma)u}\geq 0.
  \end{align*}
  This implies that
  \begin{align*}
    \ip{u}{-\div\ssigma-\bbeta\cdot\nabla u+\gamma u} = \ip{\nabla u}{\ssigma} + 
    \ip{u}{-\bbeta\cdot\nabla u+\gamma u}\geq \ip{\nabla u}{\ssigma}.
  \end{align*}
  Therefore,
  \begin{align*}
    &2\ip{u}{-\div\ssigma-\bbeta\cdot\nabla u+\gamma u} + \norm{\Amat^{1/2}\nabla u-\Amat^{-1/2}\ssigma}{}^2
    \\&\qquad\geq 2\ip{\nabla u}{\ssigma} + \norm{\Amat^{1/2}\nabla u}{}^2 + \norm{\Amat^{-1/2}\ssigma}{}^2
    - 2\ip{\nabla u}{\ssigma}
     = \norm{\Amat^{1/2}\nabla u}{}^2 + \norm{\Amat^{-1/2}\ssigma}{}^2.
  \end{align*}
  Together this yields
  \begin{align*}
    \norm{\uu}\ts^2 &= \norm{\nabla u}{}^2 + \norm{\ssigma}{}^2 + \ts\norm{\div\ssigma}{}^2
    \lesssim \norm{\nabla u}{}^2 + \norm{\ssigma}{}^2 + \ts\norm{-\div\ssigma-\bbeta\cdot\nabla u +\gamma u}{}^2  \\
    &\lesssim \norm{\Amat^{1/2}\nabla u}{}^2 + \norm{\Amat^{-1/2}\ssigma}{}^2 + \ts\norm{-\div\ssigma-\bbeta\cdot\nabla u +\gamma u}{}^2\\
    &\leq 2\ip{u}{-\div\ssigma-\bbeta\cdot\nabla u+\gamma u} + \norm{\Amat^{1/2}\nabla u-\Amat^{-1/2}\ssigma}{}^2 
    + \ts \norm{-\div\ssigma-\bbeta\cdot\nabla u+\gamma u}{}^2 = \blf(\uu,\uu).
  \end{align*}
  Coercivity of $\blfns$ follows from that of $\blf$, since
  \begin{align*}
    \blfns(\uu,\uu) \geq \tfrac12 \blf(\uu,\uu).
  \end{align*}
  This finishes the proof.
\end{proof}

For given $\uu=(u,\ssigma)\in U$ define the \emph{elliptic projector} 
$\ep_h\uu = (\ep_h^\nabla \uu,\ep_h^{\div}\uu)\in U_h$ by
\begin{align}\label{eq:def:ep}
  \blfns(\ep_h\uu,\vv_h) = \blfns(\uu,\vv_h) \quad\text{for all } \vv_h\in U_h.
\end{align}

\begin{theorem}\label{thm:ep}
  If $\uu\in U$, then $\ep_h\uu$ is well-defined.

  In particular, it holds that
  \begin{align}\label{eq:ep:opt}
    \norm{\uu-\ep_h\uu}\ts \leq C_\mathrm{opt} \inf_{\vv_h\in U_h} \norm{\uu-\vv_h}\ts.
  \end{align}
  The constant $C_\mathrm{opt}>0$ only depends on $c_\blf,C_\blf>0$.

  Moreover, if~\cref{eq:regularity} holds, then
  \begin{align}\label{eq:ep:duality}
    \norm{u-\ep_h^\nabla \uu}{} \leq C_{L^2} h \norm{\uu-\ep_h\uu}\ts.
  \end{align}
  The constant $C_{L^2}>0$ only depends on $c_\blf,C_\blf>0$, shape-regularity of $\TT$, the constant 
  in~\cref{eq:regularity}, $\bbeta,\gamma$, $T$, and $\Omega$.
\end{theorem}
\begin{proof}
  Note that~\cref{lem:blf} together with the Lax-Milgram theory show that $\ep_h\uu$ is well-defined and that it
  holds~\cref{eq:ep:opt}.

  To prove~\cref{eq:ep:duality} we need to develop some duality arguments. For optimal $L^2$ error estimates for
  least-squares finite element methods we refer to~\cite{LSQduality}.
  Note that in our case the bilinear form $\blfns$ is not symmetric and thus does not directly
  correspond to a least-squares method.
  We divide the proof into several steps.
  Throughout we use $\uu_d := (u_d,\ssigma_d) := \uu-\ep_h \uu$.

  \noindent
  \textbf{Step 1.}
  Define the function $w\in H_0^1(\Omega)$ through the unique solution of the PDE
  \begin{align*}
    -\div(\Amat\nabla w -\bbeta w) + \gamma w = u_d.
  \end{align*}
  By assumption~\cref{eq:regularity:a} there holds $w\in H^2(\Omega)$ and
  \begin{align}\label{eq:reg}
    \norm{w}{H^{2}(\Omega)} \lesssim \norm{u_d}{}.
  \end{align}
  In particular, we have that
  \begin{align}\label{eq:dual}
    \begin{split}
      \norm{u_d}{}^2 &= \ip{u_d}{-\div(\Amat\nabla w -\bbeta w) + \gamma w} = \ip{\nabla u_d}{\Amat\nabla w-\bbeta w} +
      \ip{u_d}{\gamma w} \\
      &= \ip{\Amat^{1/2}\nabla u_d}{\Amat^{1/2}\nabla w} + \ip{-\bbeta\cdot\nabla u_d+\gamma u_d}{w}\\
      &= \ip{\Amat^{1/2}\nabla u_d-\Amat^{-1/2}\ssigma_d}{\Amat^{1/2}\nabla w} + \ip{-\div\ssigma_d-\bbeta\cdot\nabla u_d+\gamma u_d}{w},
    \end{split}
  \end{align}
  where we used integration by parts and $0=\ip{-\Amat^{-1/2}\ssigma}{\Amat^{1/2}\nabla w} - \ip{\div\ssigma}w$ in the last step.

  \noindent
  \textbf{Step 2.} We will construct a function $\vv = (v,\ttau)\in U$ such that
  \begin{align}\label{eq:dual:identity}
    \norm{u_d}{}^2 = \blfns(\uu_d,\vv).
  \end{align}
  To that end, let $v\in H_0^1(\Omega)$ be the unique solution of
  \begin{align}\label{eq:dual:v}
    -\div\Amat\nabla v - \bbeta\cdot\nabla v + \gamma v + \frac1\ts v = -\div\Amat\nabla w + \frac1\ts w.
  \end{align}
  From the first step and~\cref{eq:regularity:b} we conclude that the right-hand side is in $L^2(\Omega)$.
  Then, by assumption~\cref{eq:regularity:a} there holds $v\in H^2(\Omega)$. If we define $\ttau$ by
  \begin{subequations}\label{eq:dual:fo}
  \begin{align}\label{eq:dual:fo:1}
    \Amat^{1/2}\nabla v-\Amat^{-1/2}\ttau &= \Amat^{1/2}\nabla w,
  \end{align}
  then it follows from~\cref{eq:dual:v} that
  \begin{align}
    -\div\ttau - \bbeta\cdot\nabla v + \gamma v + \frac1\ts v &= \frac1\ts w,
  \end{align}
  \end{subequations}
  and hence $\ttau\in\Hdivset\Omega$. Using~\cref{eq:dual:fo} in~\cref{eq:dual} finally shows~\cref{eq:dual:identity}.

  \noindent
  \textbf{Step 3.}
  We will show that
  \begin{align*}
    \norm{v}{H^2(\Omega)} + \norm{\ttau}{\HH^1(\Omega)} \lesssim \norm{u_d},
  \end{align*}
  Note that because $\ts$ can be arbitrary small, we can not extract such an $\ts$-independent 
  estimate directly from~\cref{eq:dual:v}.
  We therefore make the ansatz $v = \widetilde v + w$ with $\widetilde v\in H_0^1(\Omega)$.
  From~\cref{eq:dual:v} it follows that
  \begin{align*}
    -\div\Amat\nabla \widetilde v - \bbeta\cdot\nabla \widetilde v + \gamma \widetilde v + \frac1\ts \widetilde v
    - \bbeta\cdot\nabla w + \gamma w = 0.
  \end{align*}
  Multiplying by $\ts$ and simplifying yields
  \begin{align}\label{eq:dual:vTilde}
    \ts(-\div\Amat\nabla \widetilde v - \bbeta\cdot\nabla \widetilde v + \gamma \widetilde v) + \widetilde v
    = \ts(\bbeta\cdot\nabla w - \gamma w).
  \end{align}
  Testing with $\widetilde v$, using that $\ip{-\bbeta\cdot\nabla \widetilde v+\gamma \widetilde v}{\widetilde v}\geq 0$,
  and employing the bound~\cref{eq:reg}, we infer
  \begin{align*}
    \ts \norm{\nabla\widetilde v}{}^2 + \norm{\widetilde v}{}^2 
    \lesssim \ts\norm{\bbeta\cdot\nabla w - \gamma w}{}\norm{\widetilde v}{} \leq \ts\norm{u_d}{}
    (\norm{\widetilde v}{} + \ts^{1/2}\norm{\nabla \widetilde v}{}).
  \end{align*}
  This shows that
  \begin{align}\label{eq:bound:ts}
    \norm{\nabla \widetilde v}{} \leq \ts^{1/2}\norm{u_d}{} \quad\text{and}\quad \norm{\widetilde v}{} \lesssim \ts
    \norm{u_d}{}.
  \end{align}
  With $v=w+\widetilde v$ we rewrite the first-order system~\cref{eq:dual:fo} as
  \begin{align*}
    \Amat^{1/2}\nabla\widetilde v - \Amat^{-1/2}\ttau &= 0, \\
    -\div\ttau-\bbeta\cdot\nabla\widetilde v + \gamma \widetilde v 
    &= \bbeta\cdot\nabla w - \gamma w - \frac1\ts \widetilde v.
  \end{align*}
  By our regularity assumptions~\cref{eq:regularity} and the bounds~\cref{eq:reg} and~\cref{eq:bound:ts} we conclude
  \begin{align*}
    \norm{\ttau}{\HH^1(\Omega)} + \norm{\widetilde v}{H^2(\Omega)} \lesssim \norm{-\bbeta\cdot\nabla w + \gamma w -
    \frac1\ts \widetilde v}{} \lesssim \norm{u_d}{}.
  \end{align*}
  With the triangle inequality and~\cref{eq:reg} we conclude this step.

  \noindent
  \textbf{Step 4.} We prove that
  \begin{align*}
    \ts^{1/2}\norm{\div(1-I^{\rm RT}_{h,0})\ttau}{} \lesssim h \norm{u_d}{}.
  \end{align*}
  Recall from Step~3 that $\div\ttau = -\bbeta\cdot\nabla v+\gamma v + \frac1\ts \widetilde v$.
  By assumption~\cref{eq:regularity:b} we have $\bbeta\in C^1(\TT)^d$, and $\gamma\in C^1(\TT)$. Therefore, $\div\ttau \in
  H^1(\TT)$. Using the commutativity property of $I^{\rm RT}_{h,0}$ we infer
  \begin{align*}
    \ts^{1/2}\norm{\div(1-I^{\rm RT}_{h,0})\ttau}{} &= \ts^{1/2}\norm{(1-\Pi_{h,p})\div\ttau}{} 
    \lesssim \ts^{1/2} h \norm{\pwnabla \div\ttau}{} \\
    &\lesssim \ts^{1/2} h (\norm{v}{H^2(\Omega)} + \frac1\ts\norm{\nabla \widetilde v}{}).
  \end{align*}
  From Step~3 we know that $\norm{v}{H^2(\Omega)}\lesssim \norm{u_d}{}$ and 
  $\norm{\nabla\widetilde v}{}\lesssim \ts^{1/2}\norm{u_d}{}$. Hence,
  \begin{align*}
    \ts^{1/2}\norm{\div(1-I^{\rm RT}_{h,0})\ttau}{} 
    \lesssim h \norm{u_d}{}.
  \end{align*}

  \noindent
  \textbf{Step 5.}
  From the definition~\cref{eq:def:ep} of $\ep_h$,
  $\uu_d = \uu-\ep_h\uu$, and boundedness of $\blfns$, it follows that
  \begin{align*}
    \norm{u-\ep_h^\nabla \uu}{}^2 = \blfns(\uu-\ep_h\uu,\vv) = \blfns(\uu-\ep_h\uu,\vv-\vv_h)
    \lesssim \norm{\uu-\ep_h\uu}\ts \norm{\vv-\vv_h}\ts
  \end{align*}
  for any $\vv_h\in U_h$. We choose $\vv_h = (I^{\rm SZ}_{h,1} v,I^{\rm RT}_{h,0}\ttau)$. 
  It remains to show $\norm{\vv-\vv_h}\ts \lesssim h \norm{u-\ep_h^\nabla \uu}{}$ to finish the proof.
  With the regularity estimates from Step~3 and the approximation properties of the operators from~\cref{sec:discrete} we infer
  \begin{align*}
    \norm{\nabla(v-v_h)}{} + \norm{\ttau-\ttau_h}{} \lesssim h (\norm{v}{H^2(\Omega)} + \norm{\ttau}{\HH^1(\Omega)}) 
    \lesssim h\norm{u-\ep_h^\nabla\uu}{}.
  \end{align*}
  Then, together with the estimate from Step~4 this shows
  \begin{align*}
    \norm{\vv-\vv_h}\ts \lesssim h \norm{u-\ep_h^\nabla \uu}{},
  \end{align*}
  which finishes the proof.
\end{proof}

The following is a simple consequence of~\cref{thm:ep,prop:approx}.
\begin{corollary}\label{cor:ep:rates}
  Suppose $u\in H^{p+2}(\Omega)\cap H_0^1(\Omega)$, $\ssigma\in \HH^{p+1}(\Omega)$ such that $\div\ssigma\in
  H^{p+1}(\TT)$ and define $\uu := (u,\ssigma) \in U$. Then,
  \begin{align*}
    \norm{\uu-\ep_h\uu}\ts \leq C h^{p+1}(\norm{u}{H^{p+2}(\Omega)} + \norm{\ssigma}{\HH^{p+1}(\Omega)} +
    \ts^{1/2}\norm{\div\ssigma}{H^{p+1}(\TT)} )
  \end{align*}
  If additionally~\cref{eq:regularity} holds true, then
  \begin{align*}
    \norm{u-\ep_h^\nabla \uu}{} \leq C h^{p+2} (\norm{u}{H^{p+2}(\Omega)} + \norm{\ssigma}{\HH^{p+1}(\Omega)} +
    \ts^{1/2}\norm{\div\ssigma}{H^{p+1}(\TT)} ).
  \end{align*}
  The constant $C>0$ depends only on the constants from~\cref{thm:ep} and~\cref{prop:approx}.
  \qed
\end{corollary}

\section{A priori analysis}\label{sec:apriori}
\subsection{Stability of discrete solutions}
Our first result treats stability of discrete solutions of~\cref{eq:euler}.
\begin{theorem}\label{thm:stability:euler}
  Let $f\in C^0([0,T];L^2(\Omega))$ and let $u_h^0\in L^2(\Omega)$ be given.
  Let $\uu_h^n=(u_h^n,\ssigma_h^n)\in U_h^n$, $n=1,\dots,N$, denote the solutions of~\cref{eq:euler}.
  The $n$-th iteration satisfies
  \begin{align*}
    \norm{u_h^n}{} \leq \sum_{j=1}^n \ts_j \norm{f^j}{} + \norm{u_h^0}{}.
  \end{align*}
\end{theorem}
\begin{proof}
  Apply the Cauchy-Schwarz estimate to get
  \begin{align*}
    |F_n(\vv;f^n,u_h^{n-1})| &\leq \ts_n^{1/2}\norm{\frac1{\ts_n}u_h^{n-1}+f^n}{} \ts_n^{1/2}\norm{\frac1{\ts_n}v_h
    -\div\ttau_h-\bbeta\cdot\nabla v_h +\gamma v_h}{} \\
    &\leq \ts_n^{1/2}\norm{\frac1{\ts_n}u_h^{n-1}+f^n}{} \, \blfTot_n(\vv_h,\vv_h)^{1/2}.
  \end{align*}
  Thus, the choice $\vv_h = \uu_h^n$ and~\cref{eq:euler} lead to
  \begin{align*}
    \blfTot_n(\uu_h^n,\uu_h^n)^{1/2} \leq \ts_n^{1/2} \norm{f^n}{} + \ts^{-1/2}\norm{u_h^{n-1}}.
  \end{align*}
  Multiplying by $\ts_n^{1/2}$ and using that $\blf(\cdot,\cdot)$ is coercive (see~\cref{lem:blf}) we infer
  \begin{align*}
    \norm{u_h^n}{} \leq \ts_n^{1/2}\blfTot_n(\uu_h^n,\uu_h^n)^{1/2} \leq \ts_n\norm{f^n}{} +
    \norm{u_h^{n-1}}{}.
  \end{align*}
  Iterating the same arguments yields the assertion.
\end{proof}

\subsection{Optimal $L^2$ estimates}\label{sec:apriori:L2}
Our second main results concerns optimal estimates in the $L^2(\Omega)$ norm.
We will assume that the initial data $u_h^0$ for problem~\cref{eq:euler} is chosen such that
\begin{align}\label{eq:choice:initDataL2}
  \norm{\ep_h^\nabla \uu^0-u_h^0}{} \leq C_0 h^{p+2},
\end{align}
where $C_0$ is some generic constant depending also on $\uu^0$
and $\ep_h=(\ep_h^\nabla,\ep_h^\div)$ is the elliptic projection operator from~\cref{sec:duality}.
For a simpler representation of the error estimates, we use the following norms in the remainder of this article:
\begin{align*}
  \enorm{(v,\ttau)}_p^2 := \norm{v}{H^{p+2}(\Omega)}^2 + \norm{\ttau}{\HH^{p+1}(\Omega)}^2 +
  \norm{\div\ttau}{H^{p+1}(\TT)}^2.
\end{align*}
Whenever such a term appears, we assume that 
\begin{align*}
  \sup_{t\in[0,T]} \enorm{(v,\ttau)}_p < \infty.
\end{align*}
This means that we implicitly assume some regularity of the function.
Similarly, when $\norm{\ddot{u}}{}$ appears we assume that $\sup_{t\in[0,T]} \norm{\ddot{u}}{} < \infty$.

\begin{theorem}\label{thm:euler:L2}
  Suppose $\ts_n = \ts>0$ for $j=1,\dots,N$ and let $\uu_h^n\in U_h^n=U_h$ denote the solution of~\cref{eq:euler}
  and let $\uu=(u,\ssigma)$ denote the solution of~\cref{eq:fo}.
  If~\cref{eq:regularity} and~\cref{eq:choice:initDataL2} hold true, then
  \begin{align*}
    \norm{u^n-u_h^n}{} \leq C(\uu)( h^{p+2} + \ts),
  \end{align*}
  where the constant $C(\uu) = C(\enorm{\uu}_p,\enorm{\dot\uu}_p,\norm{\ddot u}{},C_0)$ and $C_0$ is the constant
  from~\cref{eq:choice:initDataL2}.
\end{theorem}
\begin{proof}
  We consider the error splitting
  \begin{align*}
    u^n-u_h^n = (u^n-\ep_h^\nabla \uu^n) + (\ep_h^\nabla\uu^n-u_h^n),
  \end{align*}
  where $\ep_h\uu\in U_h$ is the elliptic projection defined and analysed in~\cref{sec:duality}.

  \noindent
  \textbf{Step 1.}
  By~\cref{cor:ep:rates} we have that
  \begin{align*}
    \norm{u^n-\ep_h^\nabla\uu^n}{} \lesssim h^{p+2} \enorm{\uu^n}_p.
  \end{align*}

  \noindent
  \textbf{Step 2.}
  Observe that $\ep_h\uu^n$ satisfies the equation
  \begin{align*}
    \frac1\ts\ip{\ep_h^\nabla \uu^n}{v_h} + \blf(\ep_h\uu^n,\vv_h) &= 
    \frac1\ts\ip{\ep_h^\nabla \uu^n}{v_h} + \ip{\ep_h^\nabla\uu^n}{-\div\ttau_h-\bbeta\cdot\nabla v_h+\gamma v_h} +
    \blfns(\ep_h\uu^n,\vv_h) \\
    &= \frac1\ts\ip{\ep_h^\nabla \uu^n}{v_h} + \ip{\ep_h^\nabla\uu^n}{-\div\ttau_h-\bbeta\cdot\nabla v_h+\gamma v_h} +
    \blfns(\uu^n,\vv_h) \\
    &= \ip{\ep_h^\nabla \uu^n+\ts(f^n-\dot{u}^n)}{\frac1\ts v_h-\div\ttau_h-\bbeta\cdot\nabla v_h+\gamma v_h}
  \end{align*}
  By~\cref{eq:euler} the discrete solution satisfies
  \begin{align*}
    \frac1\ts\ip{u_h^n}{v_h} + \blf(\uu_h^n,\vv_h) = 
    \ts\ip{\frac{u_h^{n-1}}\ts+f^n}{\frac1\ts v_h-\div\ttau_h-\bbeta\cdot\nabla v_h+\gamma v_h}.
  \end{align*}
  Then, the error equations read
  \begin{align*}
    \frac1\ts\ip{\ep_h^\nabla \uu^n-u_h^n}{v_h} + \blf(\ep_h\uu^n-\uu_h,\vv_h) = 
    \ip{\ep_h^\nabla \uu^n-u_h^{n-1}-\ts\dot{u}^n}{\frac1\ts v_h-\div\ttau_h-\bbeta\cdot\nabla v_h+\gamma v_h}.
  \end{align*}
  The right-hand side is estimated with
  \begin{align*}
  &|\ip{\ep_h^\nabla \uu^n-u_h^{n-1}-\ts\dot{u}^n}{\frac1\ts v_h-\div\ttau_h-\bbeta\cdot\nabla v_h+\gamma v_h}| \\ 
  &\qquad  \leq \ts^{-1/2}\norm{\ep_h^\nabla \uu^n-u_h^{n-1}-\ts\dot{u}^n}{} \ts^{1/2}\norm{\frac1\ts
      v_h-\div\ttau_h-\bbeta\cdot\nabla v_h+\gamma v_h}{} \\
      &\qquad \leq \ts^{-1/2}\norm{\ep_h^\nabla \uu^n-u_h^{n-1}-\ts\dot{u}^n}{} \blfTot(\vv_h,\vv_h)^{1/2}.
  \end{align*}
  Choosing $\vv_h = \ep_h\uu^n-\uu_h^n$, this shows that
  \begin{align*}
    \blfTot(\vv_h,\vv_h) \leq \ts^{-1/2}\norm{\ep_h^\nabla \uu^n-u_h^{n-1}-\ts\dot{u}^n}{}
    \blfTot(\vv_h,\vv_h)^{1/2}.
  \end{align*}
  Dividing by $\blfTot(\vv_h,\vv_h)^{1/2}$, using coercivity of $\blf$ and multiplying by $\ts^{1/2}$, 
  yield 
  \begin{align*}
    \norm{\ep_h^\nabla \uu^n-u_h^n}{} &\leq \norm{\ep_h^\nabla \uu^n-u_h^{n-1}-\ts\dot{u}^n}{} 
    \leq \norm{\ep_h^\nabla\uu^{n-1}-u_h^{n-1}}{} + \norm{\ep_h^\nabla (\uu^n-\uu^{n-1})-\ts\dot{u}^n}{}
  \end{align*}

  \noindent
  \textbf{Step 3.}
  We follow~\cite{thomee} and rewrite 
  \begin{align*}
    \ep_h^\nabla (\uu^n-\uu^{n-1})-\ts\dot{u}^n = [\ep_h^\nabla(\uu^n-\uu^{n-1})-(u^n-u^{n-1})] +
    [u^n-u^{n-1}-\ts\dot{u}^n] =: e_h^{n,1}+e_h^{n,2}.
  \end{align*}

  \noindent
  \textbf{Step 4.}
  We estimate $\norm{e_h^{n,1}}{}$:
  Writing $\uu^n-\uu^{n-1}=\int_{t_{n-1}}^{t_n} \dot{\uu}\,ds$ and applying~\cref{cor:ep:rates} give us
  \begin{align*}
    \norm{e_h^{n,1}}{} &\leq \int_{t_{n-1}}^{t_n} \norm{\ep_h^\nabla\dot{\uu}-\dot{u}}{} \,ds \lesssim 
    h^{p+2} \int_{t_{n-1}}^{t_n} (\norm{\dot{u}}{H^{p+2}(\Omega)} + \norm{\dot\ssigma}{H^{p+1}(\Omega)} + 
    \ts^{1/2}\norm{\div\dot{\ssigma}}{H^{p+1}(\TT)} ) \,ds.
  \end{align*}

  \noindent
  \textbf{Step 5.}
  The representation $e_h^{n,2} = u^n-u^{n-1}-\ts \dot{u}^n = \int_{t_{n-1}}^{t_n}(t_{n-1}-s)\ddot{u}\,ds$ shows that
  \begin{align*}
    \norm{e_h^{n,2}}{} \leq \ts \int_{t_{n-1}}^{t_n} \norm{\ddot{u}}{} \,ds.
  \end{align*}

  \noindent
  \textbf{Step 6.}
  Putting altogether we infer that
  \begin{align*}
    \norm{\ep_h^\nabla \uu^n-u_h^n}{} &\leq \norm{\ep_h^\nabla\uu^{n-1}-u_h^{n-1}}{} + \norm{e_h^{n,1}}{} 
    + \norm{e_h^{n,2}}{} \\
    &\leq \norm{\ep_h^\nabla\uu^{n-1}-u_h^{n-1}}{} 
    + C h^{p+2} \int_{t_{n-1}}^{t_n} \enorm{\dot\uu}_p \,ds 
    + \ts \int_{t_{n-1}}^{t_n} \norm{\ddot{u}}{} \,ds.
  \end{align*}
  Iterating the above arguments and estimate~\cref{eq:choice:initDataL2} finish the proof.
\end{proof}

\subsection{Optimal error estimate in the natural norm}
For the proof of our next main result we need the following version of the well-known discrete Gronwall inequality, cf.~\cite[Lemma~10.5]{thomee}.
\begin{lemma}\label{lem:gronwall}
  Let $\alpha_n,\beta_n,\gamma_n$ be sequences, where $\beta_n$ is non-decreasing, $\gamma_n\geq 0$.
  If 
  \begin{align*}
    \alpha_n \leq \beta_n + \sum_{j=0}^{n-1} \gamma_j \alpha_j \quad\text{for all } n=1,\dots,N,
  \end{align*}
  then
  \begin{align*}
    \alpha_n \leq \beta_n e^{\sum_{j=0}^{n-1}\gamma_j} \quad\text{for all }n=1,\dots,N.
  \end{align*}
  \qed
\end{lemma}

We will assume that the initial data $u_h^0$ for problem~\cref{eq:euler} is chosen such that
\begin{align}\label{eq:choice:initDataEnergy}
  \norm{\nabla(\ep_h^\nabla \uu^0-u_h^0)}{} \leq C_0 h^{p+1},
\end{align}
where $C_0$ is some generic constant depending also on $\uu^0$
and $\ep_h=(\ep_h^\nabla,\ep_h^\div)$ is the elliptic projection operator from~\cref{sec:duality}.
Recall the norm $\enorm\cdot_p$ from~\cref{sec:apriori:L2}.

\begin{theorem}\label{thm:euler:energy}
  Suppose $\ts_n = \ts>0$ for $n=1,\dots,N$, let $\uu_h^n=(u_h^n,\ssigma_h^n)\in U_h^n=U_h$
  denote the solution of~\cref{eq:euler} and let $\uu=(u,\ssigma)$ denote the solution of~\cref{eq:fo}.
  If $u_h^0$ satisfies~\cref{eq:choice:initDataEnergy}, then
  \begin{align*}
    \norm{\uu^n-\uu_h^n}\ts \leq C(\uu)( h^{p+1} + \ts),
  \end{align*}
  where $C(\uu) = C(\enorm{\uu}_p,\enorm{\dot\uu}_p,\norm{\ddot u}{},C_0)$ and $C_0$ is the constant
  from~\cref{eq:choice:initDataEnergy}.
\end{theorem}
\begin{proof}
  We consider the error splitting
  \begin{align*}
    \uu^n-\uu_h^n = (\uu^n-\ep_h\uu^n) + (\ep_h\uu^n-\uu_h^n),
  \end{align*}
  where $\ep_h$ denotes the elliptic projection operator from~\cref{sec:duality}.

  \noindent
  \textbf{Step 1.}
  By~\cref{cor:ep:rates} we have the estimate
  \begin{align*}
    \norm{\uu^n-\ep_h\uu^n}\ts \lesssim h^{p+1} \enorm{\uu^n}_p.
  \end{align*}

  \noindent
  \textbf{Step 2.}
  If we write $\ww^n := (w^n,\cchi^n) := \ep_h\uu^n-\uu_h^n$ and $e_h^n := \ep_h^\nabla(\uu^n-\uu^{n-1})-\ts\dot{u}^n$,
  then the error equations from the proof of~\cref{thm:euler:L2} read
  \begin{align}\label{eq:euler:H1:a}
  \begin{split}
    \frac1\ts\ip{w^n}{v_h} + \blf(\ww^n,\vv_h) 
    = \frac1\ts \ip{e_h^n+w^{n-1}}{v_h+\ts(-\div\ttau_h-\bbeta\cdot\nabla v_h+\gamma v_h)}.
  \end{split}
  \end{align}
  We test this identity with $\vv_h = (v_h,\ttau_h) = (v_h,0)$. First note that by integration by parts,
  \begin{align*}
    \blf(\ww^n,\vv_h) &= \ip{w^n}{-\bbeta\cdot\nabla v_h + \gamma v_h} + \ip{-\div\cchi^n-\bbeta\cdot\nabla w^n +\gamma w^n}{v_h}
    \\ &\quad + \ts \ip{-\div\cchi^n-\bbeta\cdot\nabla w^n+\gamma w^n}{-\bbeta\cdot\nabla v_h + \gamma v_h} 
    \\ &\quad+ \ip{\Amat^{1/2}\nabla w^n-\Amat^{-1/2}\cchi^n}{\Amat^{1/2}\nabla v_h}
    \\
    & = \ip{w^n}{-\bbeta\cdot\nabla v_h + \gamma v_h} + \ip{-\bbeta\cdot\nabla w^n +\gamma w^n}{v_h}
    \\
    &\quad + \ts \ip{-\bbeta\cdot\nabla w^n+\gamma w^n}{-\bbeta\cdot\nabla v_h + \gamma v_h}
    + \ip{\Amat\nabla w^n}{\nabla v_h} - \ts\ip{\div\cchi^n}{-\bbeta\cdot\nabla v_h + \gamma v_h}.
  \end{align*}
  If we put the term $\ts\ip{\div\cchi^n}{-\bbeta\cdot\nabla v + \gamma v_h}$ to the
  right-hand side and the term $\frac1\ts \ip{w^{n-1}}v$ to the left-hand side, we obtain
  \begin{align}\label{eq:euler:H1:b}
  \begin{split}
    \frac1\ts \ip{w^n-w^{n-1}}{v_h} + c(w^n,v_h) &= \frac1\ts\ip{e_h^n}{v_h+\ts(-\bbeta\cdot\nabla v_h+\gamma v_h)} 
    \\ 
     &\qquad + \ip{w^{n-1}}{-\bbeta\cdot\nabla v_h + \gamma v_h} + \ts\ip{\div\cchi^n}{-\bbeta\cdot\nabla v_h + \gamma v_h},
  \end{split}
  \end{align}
  where we use $c: H_0^1(\Omega) \times H_0^1(\Omega)\to \R$ to denote the bilinear form defined by
  \begin{align*}
    c(u,v) &:= \ip{u}{-\bbeta\cdot\nabla v + \gamma v} + \ip{-\bbeta\cdot\nabla u +\gamma u}{v}  
    \\
    &\qquad + \ts \ip{-\bbeta\cdot\nabla u+\gamma u}{-\bbeta\cdot\nabla v + \gamma v}
    + \ip{\Amat\nabla u}{\nabla v} \quad\text{for all } u,v \in H_0^1(\Omega).
  \end{align*}
  Using $\ip{-\bbeta\cdot\nabla u+\gamma u}u \geq 0$, 
  it is straightforward to check that $c(\cdot,\cdot)$ defines a coercive and bounded bilinear form on $H_0^1(\Omega)$.
  In particular, $c(u,u) \simeq \norm{\nabla u}{}^2$ for all $u\in H_0^1(\Omega)$, where the equivalence constants
  depend only on $T$ and the coefficients $\Amat$, $\bbeta$, $\gamma$ but are otherwise independent of $\ts$.

  \noindent
  \textbf{Step 3.}
  The bilinear form $c(\cdot,\cdot)$ is symmetric and coercive on $H_0^1(\Omega)$.
  In particular, it defines a scalar product that induces the norm $\enorm\cdot$. 
  Observe that for $w,v\in H_0^1(\Omega)$ we have
  \begin{align*}
    c(w,w-v) = \frac12( \enorm{w}^2-\enorm{v}^2 + \enorm{w-v}^2).
  \end{align*}

  \noindent
  \textbf{Step 4.}
  We estimate the three terms on the right-hand side of~\cref{eq:euler:H1:b} separately.
  Throughout this step, let $\varepsilon>0$ denote the parameter for Young's inequality, which will be fixed later on.
  First,
  \begin{align*}
  \frac1\ts \ip{e_h^n}{v_h+\ts(-\bbeta\cdot\nabla v_h+\gamma v_h)} &\leq \frac{\varepsilon^{-1}}{2\ts}
  \norm{e_h^n}{}^2 + \frac\varepsilon\ts\norm{v_h}{}^2 + \varepsilon\ts\norm{-\bbeta\cdot\nabla v_h+\gamma v_h}{}^2 
  \\ &\leq \frac{\varepsilon^{-1}}{2\ts} \norm{e_h^n}{}^2 + \varepsilon\left(\frac1\ts\norm{v_h}{}^2 
  +\enorm{v_h}^2\right).
  \end{align*}
  Second, integration by parts and standard estimates show that
  \begin{align*}
    \ip{w^{n-1}}{-\bbeta\cdot\nabla v_h + \gamma v_h} &= \ip{(\div\bbeta+\gamma) w^{n-1}+\bbeta\cdot\nabla w^{n-1}}{v_h}
    \leq C_1 \varepsilon^{-1}\ts \norm{\nabla w^{n-1}}{}^2 + \frac\varepsilon{2\ts} \norm{v_h}{}^2.
  \end{align*}
  Third,
  \begin{align*}
    \ts\ip{\div\cchi^n}{-\bbeta\cdot\nabla v_h+\gamma v_h} \leq \varepsilon^{-1}\ts^2\norm{\div\cchi^n}{}^2 
    + C_2 \eps \enorm{v_h}^2.
  \end{align*}

  \noindent
  \textbf{Step 5.}
  We estimate $\ts^2\norm{\div\cchi^n}{}^2$. To do so, we consider the error equations~\cref{eq:euler:H1:a} but now test
  with $\vv_h = (0,\ttau_h)$ which gives us
  \begin{align*}
    &\ip{w^n}{-\div\ttau_h} + \ts\ip{-\div\cchi^n-\bbeta\cdot\nabla w^n + \gamma w^n}{-\div\ttau_h}
    + \ip{\Amat^{1/2}\nabla w^n-\Amat^{-1/2}\cchi^n}{-\Amat^{-1/2}\ttau_h} 
    \\ &\qquad = \ip{e_h^n}{-\div\ttau_h} + \ip{w^{n-1}}{-\div\ttau_h}.
  \end{align*}
  Integration by parts (on left-hand and right-hand side) and putting the term $\ts\ip{-\bbeta\cdot\nabla w^n+\gamma
  w^n}{\div\ttau_h}$ on the right-hand side further gives
  \begin{align*}
    \ip{\Amat^{-1}\cchi^n}{\ttau_h} + \ts\ip{\div\cchi^n}{\div\ttau_h} = \ip{e_h^n}{-\div\ttau_h}
    + \ip{\nabla w^{n-1}}{\ttau_h} + \ts\ip{-\bbeta\cdot\nabla w^n+\gamma w^n}{\div\ttau_h}
  \end{align*}
  Choosing $\ttau_h= \cchi^n$ and similar estimates as in the previous step prove
  \begin{align*}
    \norm{\cchi^n}{}^2 + \ts\norm{\div\cchi^n}{}^2 \lesssim \frac1\ts \norm{e_h^n}{}^2  
    + \norm{\nabla w^{n-1}}{}^2 + \ts\norm{\nabla w^n}{}^2.
  \end{align*}
  Hence,
  \begin{align*}
    \ts^2\norm{\div\cchi^n}{}^2 \lesssim \norm{e_h^n}{}^2  
    +\ts \norm{\nabla w^{n-1}}{}^2 + \ts^2\norm{\nabla w^n}{}^2.
  \end{align*}
  The last term can be further estimated as follows: Note that using $\vv_h=\ww^n$ in~\cref{eq:euler:H1:a} 
  and standard estimates, we obtain
  \begin{align*}
    \norm{\nabla w^n}{}^2 \lesssim \frac1\ts \norm{e_h^n}{}^2 + \frac1\ts \norm{w^{n-1}}{}^2 
    \lesssim \frac1\ts \norm{e_h^n}{}^2 + \frac1\ts \norm{\nabla w^{n-1}}{}^2.
  \end{align*}
  Therefore,
  \begin{align*}
    \ts^2\norm{\div\cchi^n}{}^2 \lesssim \norm{e_h^n}{}^2  
    +\ts \norm{\nabla w^{n-1}}{}^2 + \ts^2\norm{\nabla w^n}{}^2 \lesssim \norm{e_h^n}{}^2 +\ts \norm{\nabla w^{n-1}}{}^2.
  \end{align*}
  
  \noindent
  \textbf{Step 6.}
  Take $v_h = w^n-w^{n-1}$ in~\cref{eq:euler:H1:b}. Combining this with Step~3--5 we obtain
  \begin{align*}
    \frac1\ts\norm{v_h}{}^2 
    + \frac12\left(\enorm{w^n}^2 -\enorm{w^{n-1}}^2 + \enorm{v_h}^2\right) 
    &\leq \frac{\varepsilon^{-1}}{2\ts} \norm{e_h^n}{}^2 
    + \varepsilon\left( \frac1\ts\norm{v_h}{}^2 + \enorm{v_h}^2 \right) \\
    &\qquad + C_1{\varepsilon^{-1}}\ts \norm{\nabla w^{n-1}}{}^2 + \frac{\varepsilon}{2\ts} \norm{v_h}{}^2 \\
    &\qquad + \varepsilon^{-1} C_3(\norm{e_h^n}{}^2 + \ts\norm{\nabla w^{n-1}}{}^2)
    + C_2\varepsilon \enorm{v_h}{}^2.
  \end{align*}
  We subtract terms on the right-hand side for sufficient small $\varepsilon$ which after multiplication by $2$
  leads us to
  \begin{align*}
    \enorm{w^n}^2-\enorm{w^{n-1}}^2 \leq \frac{C_4}\ts\norm{e_h^n}{}^2 
    + C_5 \ts \norm{\nabla w^{n-1}}{}^2.
  \end{align*}
  Then, iterating the above arguments and norm equivalence $\norm{\nabla(\cdot)}{}^2\leq C_6\enorm\cdot^2$ yield
  \begin{align*}
    \enorm{w^n}^2 \leq \sum_{j=1}^n \frac{C_4}\ts\norm{e_h^j}{}^2 
      + \enorm{w^0}^2
      + \sum_{j=0}^{n-1} C_5\ts\norm{\nabla w^j}{}^2
      \leq \sum_{j=1}^n \frac{C_4}\ts\norm{e_h^j}{}^2 
      + \enorm{w^0}^2
      + \sum_{j=0}^{n-1} C_6\ts\enorm{w^j}^2.
  \end{align*}
  We apply~\cref{lem:gronwall} with 
  \begin{align*}
    \alpha_n := \enorm{w^n}{}^2, \quad \beta_n:= \sum_{j=1}^n \frac{C_4}\ts\norm{e_h^j}{}^2 + \enorm{w^0}^2,
    \quad \text{and } \gamma_n := C_6 \ts,
  \end{align*}
  which proves
  \begin{align*}
    \norm{\nabla w^n}{}^2 \simeq \enorm{w^n}^2 \lesssim \sum_{j=1}^n \frac1\ts\norm{e_h^j}{}^2 + \norm{\nabla
    w^0}{}^2.
  \end{align*}
  We estimate the terms $\norm{e_h^n}{}$:
  As in the proof of~\cref{thm:euler:L2} we write
  \begin{align*}
    e_h^n := e_h^{n,1}+e_h^{n,2} := [\ep_h^\nabla(\uu^n-\uu^{n-1})-(u^n-u^{n-1})]
    + [u^n-u^{n-1}-\ts\dot{u}^n].
  \end{align*}

  \noindent
  \textbf{Step 7.}
  \Cref{cor:ep:rates} together with the Cauchy-Schwarz inequality (for the time integral) gives us
  \begin{align*}
    \norm{(1-\ep_h^\nabla)(\uu^n-\uu^{n-1})}{} &\lesssim \int_{t_{n-1}}^{t_n} h^{p+1}
    \enorm{\dot\uu}_p \,ds 
    \leq \ts^{1/2} \left( \int_{t_{n-1}}^{t_n} h^{2(p+1)} \enorm{\dot\uu}_p^2 \,ds \right)^{1/2}.
  \end{align*}
  Thus, 
  \begin{align*}
    \frac1\ts \norm{e_h^{n,1}}{}^2 \lesssim h^{2(p+1)} 
    \int_{t_{n-1}}^{t_n} \enorm{\dot\uu}_p^2 \,ds
  \end{align*}

  \noindent
  \textbf{Step 8.}
  To estimate $\norm{e_h^{n,2}}{}$ we proceed as in the proof of~\cref{thm:euler:L2} and
  similar as in Step~7 to obtain 
  \begin{align*}
    \norm{u^n-u^{n-1}-\ts\dot{u}^n}{} \leq \ts \int_{t_{n-1}}^{t_n} \norm{\ddot{u}}{} \,ds 
    \leq \ts^{3/2} \left(\int_{t_{n-1}}^{t_n} \norm{\ddot{u}(t)}{}^2 \,ds\right)^{1/2}
  \end{align*}
  and thus,
  \begin{align*}
    \frac1{\ts}\norm{e_h^{n,2}}{}^2 \leq \ts^2 \int_{t_{n-1}}^{t_n} \norm{\ddot{u}(t)}{}^2 \,ds.
  \end{align*}

  \noindent
  \textbf{Step 9.}
  Putting altogether we infer 
  \begin{align*}
    \norm{\nabla w^n}{}^2 \lesssim \sum_{j=1}^n \frac1\ts \norm{e_h^j}{}^2 + \norm{\nabla w^0}{}^2
    &\lesssim h^{2(p+1)} 
    \int_0^{t_n} \enorm{\dot\uu}_p^2 \,ds + \ts^2 \int_0^{t_n} \norm{\ddot{u}}{}^2 \,ds + \norm{\nabla w^0}{}^2.
  \end{align*}

  \noindent
  \textbf{Step 10.}
  To get an estimate for $\norm{\ww^n}\ts$ it remains to estimate $\norm{\cchi^n}{}^2
  +\ts\norm{\div\cchi^n}{}^2$. 
  With similar arguments as in Step~5 we infer with Step~9 that
  \begin{align*}
    \norm{\cchi^n}{}^2 +\ts\norm{\div\cchi^n}{}^2 \lesssim \frac1\ts \norm{e_h^n}{}^2 + \norm{\nabla w^{n-1}}{}^2
    \lesssim h^{2(p+1)} + \ts^2 + \norm{\nabla w^0}{}^2.
  \end{align*}
  The estimate~\cref{eq:choice:initDataEnergy} for $\norm{\nabla w^0}{}$ finishes the proof.
\end{proof}

\section{Extension to a different problem}\label{sec:ext}
In this section we show that the techniques developed previously also apply for different first-order systems.
We consider
\begin{subequations}\label{eq:fo:alt}
\begin{align}
  \dot{u}-\div\ssigma+\gamma u &= f, \\
  \ssigma-\Amat\nabla u + \bbeta u &= 0, \\
  u|_\Gamma &= 0.
\end{align}
\end{subequations}
Note that this problem admits a unique solution (under the same assumptions as used for~\cref{eq:fo}
and~\cref{eq:regularity:b}).

Following~\cref{sec:main} we replace the time derivative by a finite difference approximation
which leads to the time-stepping scheme 
\begin{align*}
  \frac{u-\widetilde u^{n-1}}{\ts_n} - \div\ssigma + \gamma u &= f^n, \\
  \Amat^{-1/2}\ssigma - \Amat^{1/2}\nabla u + \Amat^{-1/2}\bbeta u &= 0, \\
  u|_\Gamma &= 0,
\end{align*}
where $\widetilde u^{n-1}$ denotes an approximation of $u^{n-1}$.

In what follows we redefine the bilinear forms $\blf_n,\blfns_n,\blfTot_n$ and the functional $F_n$
from~\cref{sec:main} to account for the different first-order system~\cref{eq:fo:alt}.
We show that the main results on the bilinear forms, i.e.,~\cref{lem:blf,thm:ep} hold true.
Moreover,~\cref{thm:euler:L2} and~\cref{thm:euler:energy} transfer to the present situation following the same lines of
proof with some minor modifications.

\subsection{Backward Euler method}
Define $\blfns_n,\blf_n,\blfTot_n : U\times U \to \R$, $F_n : U\to \R$ by
\begin{subequations}\label{eq:def:blf:alt}
\begin{align}
  \blfns_n(\uu,\vv) &:= \ip{-\div\ssigma+\gamma u}{v} 
  + \ts_n \ip{-\div\ssigma+\gamma u}{-\div\ttau+\gamma v} \\
  &\qquad + \ip{\Amat^{-1/2}\ssigma-\Amat^{1/2}\nabla u+\Amat^{-1/2}\bbeta u}{\Amat^{-1/2}\ttau-\Amat^{1/2}\nabla v+\Amat^{-1/2}\bbeta v}
  \\ 
  \blf_n(\uu,\vv) &:= \ip{u}{-\div\ttau+\gamma v} + \blfns_n(\uu,\vv), \\ 
  \blfTot_n(\uu,\vv) &:= \frac1{\ts_n}\ip{u}v + \blf_n(\uu,\vv), \\
  F_n(\vv;f^n,w) &:= \ip{f^n+\frac{w}{\ts_n}}{v+\ts_n(-\div\ttau+\gamma v)}.
\end{align}
\end{subequations}
The \emph{backward Euler} method then reads: Given $u_h^0\in L^2(\Omega)$, find $\uu_h^n\in U_h^n$ such that for all
$n=1,\dots,N$ 
\begin{align}\label{eq:euler:alt}
  \blfTot_n(\uu_h^n,\vv) = F_n(\vv;f^n,u_h^{n-1}) \quad\text{for all } \vv\in U_h^n.
\end{align}

Similar as in~\cref{sec:main} we conclude:
\begin{theorem}\label{thm:euler:alt}
  For all $n=1,\dots,N$ Problem~\cref{eq:euler:alt} admits a unique solution $\uu_h^n\in U_h^n$.
  \qed
\end{theorem}

\subsection{Elliptic projection operator}
In this section, we set $\ts_n:=\ts\in(0,T]$ and skip the indices in the bilinear forms.
We define the elliptic projection operator $\ep_h : U \to U_h$ with respect to the elliptic part of the equations, which
is given by the bilinear form $\blfns$, i.e.,
\begin{align}\label{eq:def:ep:alt}
  \blfns(\ep_h\uu,\vv_h) = \blfns(\uu,\vv_h) \quad\text{for all } \vv_h\in U_h.
\end{align}

\begin{lemma}
  \cref{lem:blf} holds true for the bilinear forms defined in~\cref{eq:def:blf:alt}.
\end{lemma}
\begin{proof}
  Boundedness with respect to the norm $\norm\cdot{\ts}$ follows with the same argumentation as in the proof
  of~\cref{lem:blf}.

  It remains to prove coercivity. We show the coercivity of $\blf(\cdot,\cdot)$ only. Coercivity of
  $\blfns(\cdot,\cdot)$ then follows since $\blfns(\uu,\uu)\geq \tfrac12\blf(\uu,\uu)$.

  Observe that
  \begin{align*}
    &2\ip{u}{-\div\ssigma+\gamma u} + \norm{\Amat^{-1/2}\ssigma-\Amat^{1/2}\nabla u+\Amat^{-1/2}\bbeta u}{}^2 \\
    &= 2\ip{\ssigma}{\nabla u} + 2\ip{\gamma u}{u} + \norm{\Amat^{-1/2}\ssigma-\Amat^{1/2}\nabla u}{}^2
    + 2\ip{\Amat^{-1/2}\ssigma-\Amat^{1/2}\nabla u}{\Amat^{-1/2}\bbeta u} + \norm{\Amat^{-1/2}\bbeta u}{}^2 \\
    &= 2\ip{-\bbeta\cdot\nabla u+\gamma u}u + \norm{\Amat^{-1/2}\ssigma}{}^2 
    + \norm{\Amat^{1/2}\nabla u}{}^2 + 2\ip{\Amat^{-1/2}\ssigma}{\Amat^{-1/2}\bbeta u} + \norm{\Amat^{-1/2}\bbeta u}{}^2
  \end{align*}
  Let $1\geq \mu>0$. Young's inequality with parameter $\varepsilon>0$ leads us to 
  \begin{align*}
    2|\ip{\Amat^{-1/2}\ssigma}{\Amat^{-1/2}\bbeta u}| &\leq \varepsilon\norm{\Amat^{-1/2}\ssigma}{}^2
    + \varepsilon^{-1}\norm{\Amat^{-1/2}\bbeta u}{}^2 \\
    &= \varepsilon\norm{\Amat^{-1/2}\ssigma}{}^2 + \varepsilon^{-1}\mu\norm{\Amat^{-1/2}\bbeta u}{}^2 +
    \varepsilon^{-1}(1-\mu)\norm{\Amat^{-1/2}\bbeta u}{}^2   \\
    &\leq \varepsilon\norm{\Amat^{-1/2}\ssigma}{}^2 + C\varepsilon^{-1}\mu\norm{\Amat^{1/2}\nabla u}{}^2 +
    \varepsilon^{-1}(1-\mu)\norm{\Amat^{-1/2}\bbeta u}{}^2,
  \end{align*}
where $C>0$ depends on $\Amat$, $\bbeta$ and $\CF$.
  Together with $\ip{-\bbeta\cdot\nabla u+\gamma u}u\geq 0$ this further shows
  \begin{align*}
    &2\ip{u}{-\div\ssigma+\gamma u} + \norm{\Amat^{-1/2}\ssigma-\Amat^{1/2}\nabla u+\Amat^{-1/2}\bbeta u}{}^2 \\
    &\qquad \geq (1-\varepsilon)\norm{\Amat^{-1/2}\ssigma}{}^2 
    + (1-C\varepsilon^{-1}\mu)\norm{\Amat^{1/2}\nabla u}{}^2 + (1-\varepsilon^{-1}(1-\mu))\norm{\Amat^{-1/2}\bbeta u}{}^2.
  \end{align*}
  Choose $\mu\in(0,1]$ sufficiently small such that $C\mu<1-\mu$ and choose $\varepsilon$ with $1-\mu<\varepsilon<1$.
  This implies that
  \begin{align*}
    1-\varepsilon>0, \quad 1-\varepsilon^{-1}\mu C > 0, \quad\text{and } 1-\varepsilon^{-1}(1-\mu)>0,
  \end{align*}
  hence,
  \begin{align*}
    &2\ip{u}{-\div\ssigma+\gamma u} + \norm{\Amat^{-1/2}\ssigma-\Amat^{1/2}\nabla u+\Amat^{-1/2}\bbeta u}{}^2
    \gtrsim  \norm{\Amat^{-1/2}\ssigma}{}^2 + \norm{\Amat^{1/2}\nabla u}{}^2 + \norm{\Amat^{-1/2}\bbeta u}{}^2.
  \end{align*}
  Finally, with the triangle inequality we infer that
  \begin{align*}
    \norm{\uu}\ts^2 \lesssim \norm{\Amat^{-1/2}\ssigma}{}^2 + \norm{\Amat^{1/2}\nabla u}{}^2 + \norm{\Amat^{-1/2}\bbeta u}{}^2 
    + \ts \norm{-\div\ssigma+\gamma u}{}^2
    \lesssim \blf(\uu,\uu),
  \end{align*}
  which finishes the proof.
\end{proof}

\begin{theorem}
  \cref{thm:ep} holds for the operator defined in~\cref{eq:def:ep:alt}.
\end{theorem}
\begin{proof}
  It suffices to develop duality arguments.
  Let $\uu_d =(u_d,\ssigma_d) = \uu-\ep_h\uu$, where $\ep_h\uu$ is defined in~\cref{eq:def:ep:alt}.
  Define $w\in H_0^1(\Omega)$ as the solution of
  \begin{align*}
    -\div\Amat\nabla w -\bbeta\cdot\nabla w +\gamma w = u_d.
  \end{align*}
  Then,
  \begin{align*}
    \norm{u_d}{}^2 &= \ip{u_d}{-\div\Amat\nabla w -\bbeta\cdot\nabla w +\gamma w} 
    = \ip{\Amat\nabla u_d -\bbeta u_d}{\nabla w} + \ip{\gamma u_d}{w}.
  \end{align*}
  Adding $-\ip{\ssigma_d}{\nabla w} - \ip{\div\ssigma_d}w=0$ we infer
  \begin{align*}
    \norm{u_d}{}^2 &= \ip{\Amat^{1/2}\nabla u_d -\Amat^{-1/2}\bbeta u_d-\Amat^{-1/2}\ssigma_d}{\Amat^{1/2}\nabla w}
    + \ts \ip{-\div\ssigma_d+\gamma u_d}{\frac{w}\ts}.
  \end{align*}
  Define $(v,\ttau)\in U$ via
  \begin{align*}
    -\div\ttau +\gamma v + \frac{v}\ts &= \frac{w}\ts, \\
    \Amat^{1/2}\nabla v - \Amat^{-1/2}\bbeta v -\Amat^{-1/2}\ttau &= \Amat^{1/2}\nabla w.
  \end{align*}
  Therefore, $v\in H_0^1(\Omega)$ satisfies
  \begin{align*}
    -\div\Amat\nabla v + \div(\bbeta v) + \gamma v + \frac{v}\ts &= -\div\Amat\nabla w + \frac{w}\ts.
  \end{align*}
  Following the same argumentation as in the proof of~\cref{thm:ep} (Step~3--Step~5) we conclude the proof.
\end{proof}

\begin{corollary}
  \cref{cor:ep:rates} holds for the operator defined in~\cref{eq:def:ep:alt}. \qed
\end{corollary}

\subsection{Error estimates}
The main results from~\cref{sec:apriori} hold true for the present situation:
\begin{theorem}\label{thm:euler:alt:main}
  \cref{thm:stability:euler,thm:euler:L2,thm:euler:energy} hold for solutions of~\cref{eq:euler:alt}.
\end{theorem}
\begin{proof}
  Following the same lines as in the proof of~\cref{thm:stability:euler,thm:euler:L2} we conclude the
  analogous results for solutions of~\cref{eq:euler:alt}.

  For the proof of the analogous result of~\cref{thm:euler:energy} we need some minor modifications
  of the proof of~\cref{thm:euler:energy}, but the main ideas are the same:
  With the same notation as in the proof of~\cref{thm:euler:energy}, the error equations read
  \begin{align}\label{eq:euler:alt:H1:a}
    \frac1\ts \ip{w^n}{v_h} + b(\ww^n,\vv_h) &= \frac1\ts\ip{e_h^n}{v_h+\ts(-\div\ttau_h + \gamma v_h)} 
   + \frac1\ts\ip{w_h^{n-1}}{v_h+\ts(-\div\ttau_h+\gamma v_h)}.
  \end{align}
  Taking $\vv_h = (v_h,0)$ we obtain (using integration by parts)
  \begin{align*}
    b(\ww^n,\vv_h) &= \ip{-\div\cchi^n+\gamma w^n}{v_h} + \ip{w^n}{\gamma v_h} + \ts\ip{-\div\cchi^n+\gamma
    w^n}{\gamma v_h}
    \\ &\qquad + \ip{\Amat^{-1/2}\cchi^n-\Amat^{1/2}\nabla w^n+\Amat^{-1/2}\bbeta
    w^n}{-\Amat^{1/2}\nabla v_h+\Amat^{-1/2}\bbeta v_h}
    \\ &= 2\ip{\gamma w^n}{v_h} + \ts \ip{\gamma^2 w^n}{v_h} + \ip{\Amat\nabla w^n}{\nabla v_h}
    -\ip{\nabla w^n}{\bbeta v_h} - \ip{\bbeta w^n}{\nabla v_h} 
    \\ &\qquad + \ip{\Amat^{-1}\bbeta w^n}{\bbeta v_h} + \ip{\Amat^{-1}\cchi^n}{\bbeta v_h} -\ts
    \ip{\div\cchi^n}{\gamma v_h}
  \end{align*}
  Define the symmetric, bounded and coercive bilinear form $c : H_0^1(\Omega)\times H_0^1(\Omega) \to \R$ by
  \begin{align*}
    c(u,v) := 2\ip{\gamma u}{v} + \ts \ip{\gamma^2 u}{v} + \ip{\Amat\nabla u}{\nabla v}
    -\ip{\nabla u}{\bbeta v} - \ip{\bbeta u}{\nabla v} 
    + \ip{\Amat^{-1}\bbeta u}{\bbeta v}
  \end{align*}
  for $u,v\in H_0^1(\Omega)$. Then,~\cref{eq:euler:alt:H1:a} with $\vv_h = (v_h,0)$ reads
  \begin{align*}
    \frac1\ts \ip{w^n-w^{n-1}}{v_h} + c(w^n,v_h) &= \frac1\ts \ip{e_h^n}{v_h+\ts\gamma v_h}
    + \ip{w^{n-1}}{\gamma v_h} \\
    &\qquad -\ip{\Amat^{-1}\cchi^n}{\bbeta v_h} + \ts \ip{\div\cchi^n}{\gamma v_h}.
  \end{align*}
  With $\enorm{\cdot}^2 := c(\cdot,\cdot)$ and $v_h = w^n-w^{n-1}$
  we conclude with the same arguments as in the proof of~\cref{thm:euler:energy} that
  \begin{align}\label{eq:euler:alt:H1:b}
    \enorm{w^n}^2-\enorm{w^{n-1}}^2 \leq \frac{C_1}\ts \norm{e_h^n}{}^2 + C_2\ts \norm{\nabla w^{n-1}}{}^2 
    + C_3( \ts\norm{\cchi^n}{}^2 +\ts^3\norm{\div\cchi^n}{}^2).
  \end{align}
  In the following we estimate $\ts\norm{\cchi^n}{}^2 +\ts^3\norm{\div\cchi^n}{}^2$. By
  testing~\cref{eq:euler:alt:H1:a} with $\vv_h= (0,\ttau_h)$ we get
  \begin{align*}
    \ip{\Amat^{-1}\cchi^n}{\ttau_h} + \ts\ip{\div\cchi^n}{\div\ttau_h}
    &= \ip{e_h^n}{-\div\ttau_h} + \ip{w^{n-1}}{-\div\ttau_h} 
    \\ 
    &\qquad -\ip{\Amat^{-1}\bbeta w^n}{\ttau_h} 
    +\ts \ip{\gamma w^n}{\div\ttau_h}
    \\ &= \ip{e_h^n}{-\div\ttau_h} + \ip{\nabla w^{n-1}}{\ttau_h} 
    \\ 
    &\qquad -\ip{\Amat^{-1}\bbeta w^n}{\ttau_h} 
    +\ts \ip{\gamma w^n}{\div\ttau_h}.
  \end{align*}
  Then, standard arguments give
  \begin{align*}
    \norm{\cchi^n}{}^2 + \ts \norm{\div\cchi^n}{}^2 \lesssim \frac1\ts\norm{e_h^n}{}^2 + \norm{\nabla
    w^{n-1}}{}^2 + \norm{w^n}{}^2.
  \end{align*}
  Using $\norm{w^n}{}\leq \norm{w^{n-1}}{} + \norm{e_h^n}{} \lesssim \norm{\nabla w^{n-1}}{} + \norm{e_h^n}{}$ we further infer
  \begin{align*}
    \norm{\cchi^n}{}^2 + \ts \norm{\div\cchi^n}{}^2 \lesssim \frac1\ts\norm{e_h^n}{}^2 + \norm{\nabla w^{n-1}}{}^2
  \end{align*}
  and get instead of~\cref{eq:euler:alt:H1:b} the estimate
  \begin{align*}
    \enorm{w^n}^2-\enorm{w^{n-1}}^2 \leq \frac{C_4}\ts \norm{e_h^n}{}^2 + C_5\ts \norm{\nabla w^{n-1}}{}^2.
  \end{align*}
  Finally, the very same argumentation as in the proof of~\cref{thm:euler:energy} yields
  \begin{align*}
    \norm{\nabla w^n}{} \lesssim h^{p+1}+\ts  \quad\text{and also}\quad
    \norm{\ww^n}\ts \lesssim h^{p+1}+\ts.
  \end{align*}
  This finishes the proof.
\end{proof}

\section{Numerical Examples}\label{sec:num}
In this section we present some numerical examples that underpin the theoretical results uncovered in this work. In
particular, we are interested in the convergence rates predicted
by~\cref{thm:euler:L2,thm:euler:energy,thm:euler:alt:main}.

\begin{figure}
  \begin{center}
    \begin{tikzpicture}
\begin{axis}[
width=0.55\textwidth,
    axis equal,
]

\addplot[patch,color=white,
faceted color = black, line width = 1.5pt,
patch table ={figures/elements.dat}] file{figures/coordinates.dat};
\addplot[mark=*,color=gray,only marks] table[x index=0,y index=1] {figures/coordinates.dat};
\end{axis}
\end{tikzpicture}
  \end{center}
  \caption{Initial mesh configuration of $\Omega = (0,1)^2$ with four triangles.}
  \label{fig:Mesh}
\end{figure}
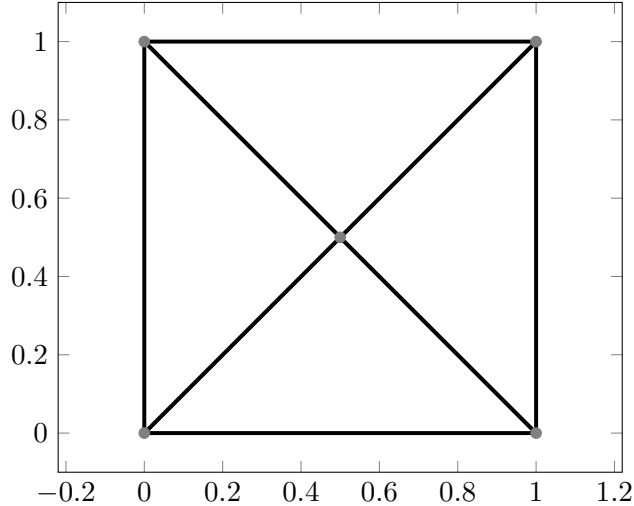

Throughout, we consider the convex domain $\Omega = (0,1)^2$ and the manufactured solution
\begin{align*}
  u(t;x,y) := e^{-2\pi^2 t}\sin(\pi x) \sin(\pi y).
\end{align*}
Note that $u$ is smooth and its trace vanishes on $\Gamma = \partial \Omega$.
Moreover, we choose $\Amat$ to be the identity matrix, $\bbeta(x,y) = (1,1)^T$, $\gamma(x,y) = 0$ for $(x,y)\in \Omega$.
In~\cref{sec:examples:1} we consider the solutions of~\cref{eq:euler}, where the right-hand side data $f$ is given
through~\cref{eq:fo}, i.e., $f := \dot{u}- \Delta u -\bbeta\cdot\nabla u$. Also note that by~\cref{eq:fo} we have $\ssigma = \nabla u$.
In~\cref{sec:examples:2} we consider the solutions of~\cref{eq:euler:alt}, where the right-hand side data $f$ is
given through~\cref{eq:fo:alt}, i.e., $f = \dot{u}-\Delta u + \bbeta\cdot\nabla u$. Furthermore, in this case
$\ssigma = \nabla u -\bbeta u$.

In the following examples we use a lowest-order discretization, i.e., $p=0$. In order to obtain the optimal possible
rates we have to consider two cases: For optimal rates for the scalar variable in the $L^2(\Omega)$ norm we have to
equilibrate the time step $\ts$ and the mesh-size $h$ so that $\ts \simeq h^2$.
If we are only interested in optimal rates for the error in the natural norm $\norm{\cdot}\ts$, it suffices to choose
$\ts \simeq h$.

The problem set-up is quite simple: We choose as end time $T=0.1$.
For all computations we start with an initial triangulation of $\Omega$ into four (similar) 
triangles, see~\cref{fig:Mesh}, and an initial time step $\ts = 0.1$.
We then solve~\cref{eq:euler}, see~\cref{sec:examples:1}, resp.~\cref{eq:euler:alt}, see~\cref{sec:examples:2}.
In the next step we refine the triangulation uniformly, i.e., each triangle is divided into four son triangles (by Newest Vertex Bisection).
In our configuration this leads to four triangles which have diameter exactly half of their father triangle,
i.e., $h/2$.
Depending on which case we consider we then halve or quarter the time step, which ensures that either $\ts\simeq h$
or $\ts\simeq h^2$.

The initial data $u_h^0$ is chosen to be the $L^2$-projection of $u^0=u(0,\cdot)$ onto $\cS_0^1(\TT)$.
Note that $u$ is smooth (for all times), so that the initial data satisfies~\cref{eq:choice:initDataL2}
and~\cref{eq:choice:initDataEnergy}.

In all figures, we plot the errors that correspond to approximations of $u^N=u(T,\cdot)$, resp.,
$\ssigma^N=\ssigma(T,\cdot)$ over the number of degrees of freedom in $U_h$.
To that end define the following error quantities
\begin{align*}
  \err(u_h^N) &:= \norm{u^N-u_h^N}{}, \\
  \err(\nabla u_h^N) &:= \norm{\nabla(u^N-u_h^N)}{}, \\
  \err(\ssigma_h^N) &:= \norm{\ssigma^N-\ssigma_h^N}{}, \\
  \err(\div\ssigma_h^N) &:= \norm{\div(\ssigma^N-\ssigma_h^N)}{}.
\end{align*}
Triangles in the plots visualize the order of convergence, i.e., their negative slope $\alpha$ corresponds to $\OO(h^\alpha)$.

\subsection{Example~1}\label{sec:examples:1}

\begin{figure}
  \begin{center}
    \begin{tikzpicture}
\begin{loglogaxis}[
width=0.65\textwidth,
cycle list/Dark2-6,
cycle multiindex* list={
mark list*\nextlist
Dark2-6\nextlist},
every axis plot/.append style={ultra thick},
xlabel={degrees of freedom},
grid=major,
legend entries={\small $\err(u_h^N)$,\small $\err(\nabla u_h^N)$,
\small $\err(\ssigma_h^N)$, \small $\err(\div\ssigma_h^N)$},
legend pos=south west,
]
\addplot table [x=nDOF,y=errL2] {figures/Example1L2.dat};
\addplot table [x=nDOF,y=errGrad] {figures/Example1L2.dat};
\addplot table [x=nDOF,y=errSigmaL2] {figures/Example1L2.dat};
\addplot table [x=nDOF,y=errDivSigma] {figures/Example1L2.dat};

\logLogSlopeTriangle{0.9}{0.2}{0.62}{0.5}{black}{{\small $1$}};
\logLogSlopeTriangleBelow{0.82}{0.2}{0.1}{1}{black}{{\small $2$}};
\end{loglogaxis}
\end{tikzpicture}
  \end{center}
  \caption{Example from~\cref{sec:examples:1} with $\ts\simeq h^2$.}
  \label{fig:Example1L2}
\end{figure}
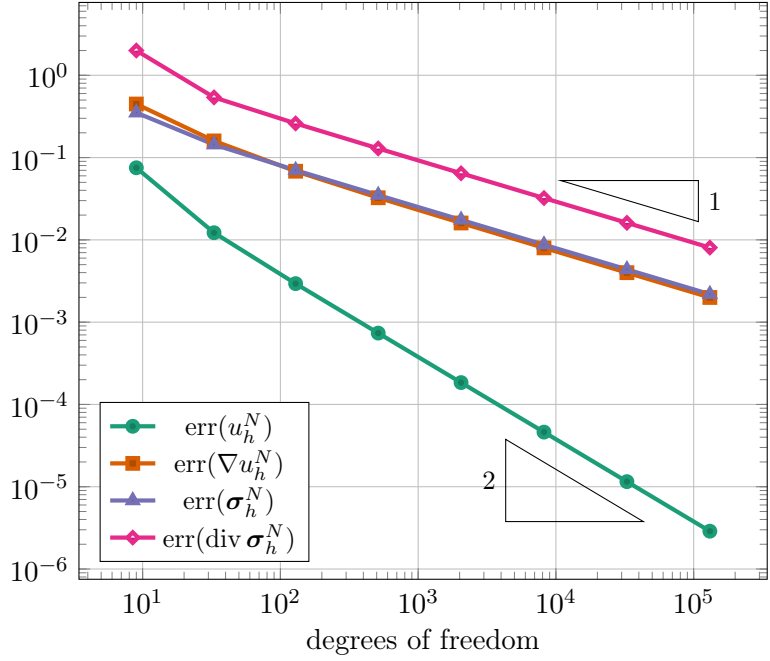

\begin{figure}
  \begin{center}
    \begin{tikzpicture}
\begin{loglogaxis}[
width=0.65\textwidth,
cycle list/Dark2-6,
cycle multiindex* list={
mark list*\nextlist
Dark2-6\nextlist},
every axis plot/.append style={ultra thick},
xlabel={degrees of freedom},
grid=major,
legend entries={\small $\err(u_h^N)$,\small $\err(\nabla u_h^N)$,
\small $\err(\ssigma_h^N)$, \small $\err(\div\ssigma_h^N)$},
legend pos=south west,
]
\addplot table [x=nDOF,y=errL2] {figures/Example1Energy.dat};
\addplot table [x=nDOF,y=errGrad] {figures/Example1Energy.dat};
\addplot table [x=nDOF,y=errSigmaL2] {figures/Example1Energy.dat};
\addplot table [x=nDOF,y=errDivSigma] {figures/Example1Energy.dat};

\logLogSlopeTriangle{0.9}{0.2}{0.45}{0.5}{black}{{\small $1$}};
\end{loglogaxis}
\end{tikzpicture}
  \end{center}
  \caption{Example from~\cref{sec:examples:1} with $\ts\simeq h$.}
  \label{fig:Example1Energy}
\end{figure}
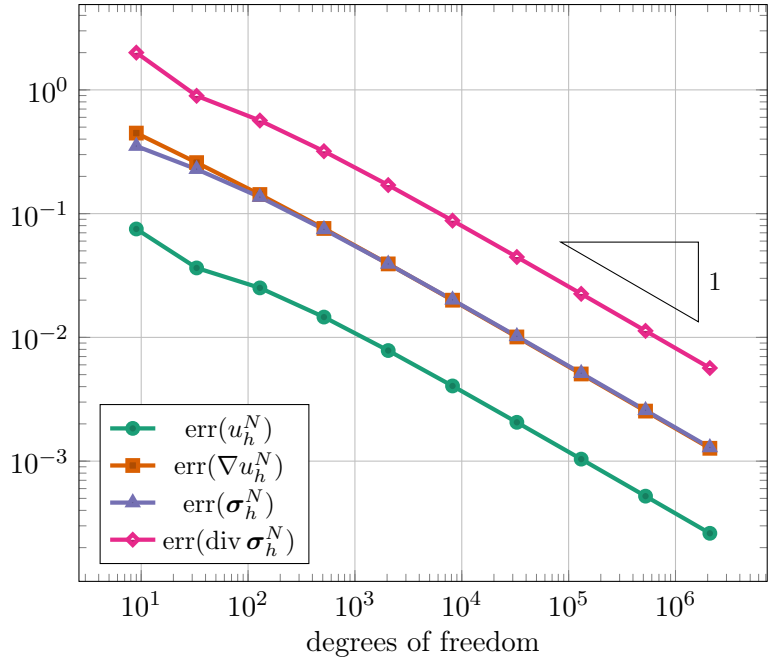

We consider the solutions $\uu_h^N=(u_h^N,\ssigma_h^N)\in U_h$ of~\cref{eq:euler}.
\Cref{fig:Example1L2} shows the error quantities in the case where $\ts\simeq h^2$. By~\cref{thm:euler:L2} we expect
that $\err(u_h^N)\lesssim h^2$. This is also observed in~\cref{fig:Example1L2}.
Note that the other quantities converge at a lower rate.
In particular, observe that $\err(\div\ssigma_h^N)$ converges at the same rate as $\err(\ssigma_h^N)$, although from
our analysis (\cref{thm:euler:energy}) we can only infer that
\begin{align*}
  \norm{\uu^N-\uu_h^N}\ts \simeq \err(\nabla u_h^N) + \err(\ssigma_h^N) + \ts^{1/2}\err(\div\ssigma_h^N) =
  \OO(h+\ts)= \OO(h+h^2) = \OO(h).
\end{align*}

The results for the case where $\ts\simeq h$ are displayed in~\cref{fig:Example1Energy}. 
We observe that all error quantities, including $\err(u_h^N) = \norm{u^N-u_h^N}{}$, converge at the same rate.
We conclude that it is not only sufficient but in general also necessary to set $\ts\simeq h^2$ to obtain higher rates
for the $L^2$ error $\err(u_h^N)$.
As before we see that $\err(\div\ssigma_h^N)$ converges even at a higher rate than expected.

\subsection{Example~2}\label{sec:examples:2}

\begin{figure}
  \begin{center}
    \begin{tikzpicture}
\begin{loglogaxis}[
width=0.65\textwidth,
cycle list/Dark2-6,
cycle multiindex* list={
mark list*\nextlist
Dark2-6\nextlist},
every axis plot/.append style={ultra thick},
xlabel={degrees of freedom},
grid=major,
legend entries={\small $\err(u_h^N)$,\small $\err(\nabla u_h^N)$,
\small $\err(\ssigma_h^N)$, \small $\err(\div\ssigma_h^N)$},
legend pos=south west,
]
\addplot table [x=nDOF,y=errL2] {figures/Example2L2.dat};
\addplot table [x=nDOF,y=errGrad] {figures/Example2L2.dat};
\addplot table [x=nDOF,y=errSigmaL2] {figures/Example2L2.dat};
\addplot table [x=nDOF,y=errDivSigma] {figures/Example2L2.dat};

\logLogSlopeTriangle{0.9}{0.2}{0.62}{0.5}{black}{{\small $1$}};
\logLogSlopeTriangleBelow{0.82}{0.2}{0.1}{1}{black}{{\small $2$}};
\end{loglogaxis}
\end{tikzpicture}
  \end{center}
  \caption{Example from~\cref{sec:examples:2} with $\ts\simeq h^2$.}
  \label{fig:Example2L2}
\end{figure}

\begin{figure}
  \begin{center}
    \begin{tikzpicture}
\begin{loglogaxis}[
width=0.65\textwidth,
cycle list/Dark2-6,
cycle multiindex* list={
mark list*\nextlist
Dark2-6\nextlist},
every axis plot/.append style={ultra thick},
xlabel={degrees of freedom},
grid=major,
legend entries={\small $\err(u_h^N)$,\small $\err(\nabla u_h^N)$,
\small $\err(\ssigma_h^N)$, \small $\err(\div\ssigma_h^N)$},
legend pos=south west,
]
\addplot table [x=nDOF,y=errL2] {figures/Example2Energy.dat};
\addplot table [x=nDOF,y=errGrad] {figures/Example2Energy.dat};
\addplot table [x=nDOF,y=errSigmaL2] {figures/Example2Energy.dat};
\addplot table [x=nDOF,y=errDivSigma] {figures/Example2Energy.dat};

\logLogSlopeTriangle{0.9}{0.2}{0.45}{0.5}{black}{{\small $1$}};
\end{loglogaxis}
\end{tikzpicture}
  \end{center}
  \caption{Example from~\cref{sec:examples:2} with $\ts\simeq h$.}
  \label{fig:Example2Energy}
\end{figure}

We consider the solutions $\uu_h^N=(u_h^N,\ssigma_h^N)\in U_h$ from~\cref{eq:euler:alt}.
\Cref{fig:Example2L2,fig:Example2Energy} show the error quantities in the cases $\ts\simeq h^2$ and $\ts\simeq h$,
respectively.
We make the same observations as in~\cref{sec:examples:1}.

\bibliographystyle{abbrv}
\bibliography{literature}

\end{document}